\documentclass[smallextended]{svjour3}       % onecolumn (second format)
\smartqed  % flush right qed marks, e.g. at end of proof
\usepackage{graphicx}
\usepackage{ulem}%% used for comments
\usepackage{subfigure}
\usepackage{graphicx}
\usepackage{enumerate}
\usepackage{amsmath,bm}
\usepackage{latexsym}
\usepackage{float}
\usepackage{latexsym}
\usepackage{amsmath}
\usepackage{amssymb}
\usepackage{amsfonts}
\usepackage{verbatim}
\usepackage{mathrsfs}
\usepackage{amsfonts}
\usepackage{colortbl,dcolumn}
\usepackage{amsmath}
\usepackage{amsfonts}
\usepackage{graphicx}
\usepackage{epstopdf}
\usepackage{algorithmic}

\allowdisplaybreaks[4]
\numberwithin{equation}{section}
\newtheorem{tm}{Theorem}
\newtheorem{rk}{Remark}

\newtheorem{prop}{Proposition}
\newtheorem{lm}{Lemma}
\newtheorem{cor}{Corollary}

\newcommand{\E}{\mathbb E}
\newcommand{\PP}{\mathbb P}
\newcommand{\N}{\mathbb N}
\newcommand{\R}{\mathbb R}

\newcommand{\bs}{\mathbf s}

\newcommand{\OO}{\mathcal O}

\newcommand{\HH}{\mathbb H}
\newcommand{\LL}{\mathcal L}

\newcommand{\FFF}{\mathscr F}

\newcommand{\<}{\langle}
\renewcommand{\>}{\rangle}
\allowdisplaybreaks[4]

\begin{document}

\title{Numerical analysis of a full discretization for stochastic Cahn--Hilliard equation driven by additive noise
}
%\subtitle{Analysis of a full discretization  for SCH equation}

\titlerunning{A full discretization  for SCHE driven by additive noise }        % if too long for running head

\author{Jianbo Cui        \and
        Jialin Hong \and
        Liying Sun%etc.
}


\institute{Jianbo Cui \at
             School of Mathematics, Georgia Institute of Technology, Atlanta, GA 30332, USA \\
              \email{jcui82@gatech.edu}           %  \\
           \and
           Jialin Hong \at
              1. LSEC, ICMSEC, 
Academy of Mathematics and Systems Science,\\ Chinese Academy of Sciences,
 Beijing, 100190, China\\
2. School of Mathematical Science, University of Chinese Academy of Sciences,\\ Beijing, 100049, China \\
              \email{hjl@lsec.cc.ac.cn} 
              \and
                    Liying Sun \at
              1. LSEC, ICMSEC, 
Academy of Mathematics and Systems Science,\\ Chinese Academy of Sciences,
 Beijing, 100190, China\\
2. School of Mathematical Science, University of Chinese Academy of Sciences,\\ Beijing, 100049, China \\
              \email{liyingsun@lsec.cc.ac.cn}  
}

\date{}

\maketitle

\begin{abstract}
In this article, we consider the stochastic Cahn--Hilliard equation driven by additive noise. We discretize this equation by using a spatial spectral Galerkin method and a temporal accelerated implicit Euler method. We first present the optimal regularity estimates of both the exact and numerical solutions. 
Then we prove that the proposed numerical method
is strongly convergent with the sharp convergence rate in a negative Sobolev space. 
With the help of the semigroup theory and interpolation inequality, we deduce the spatial optimal convergence rate and the temporal super-convergence rate of the proposed numerical method in strong sense.  

\end{abstract}

\section{Introduction}

This paper focuses on the numerical analysis of a full discretization for the stochastic Cahn--Hilliard equation with polynomial nonlinearity and additive noise.
The stochastic Cahn--Hilliard  equation, as a phenomenological model for studying the spinodal decomposition,  describes the evolution of a concentration.
More precisely, the concentration $X$ satisfies 
\begin{align}\label{spde}
dX(t)+A(AX(t)+f(X(t)))dt&=dW(t),\quad t\in (0,T]\\\nonumber
X(0)&=X_0,
\end{align}
with $0<T<\infty$, $\mathcal O=[0,L]^d$, $L>0,$ $d\le 3$ and $A:D(A)\subset H:=L^2(\OO) \to H$ being the Neumann Laplacian operator.
In addition, $\{W(t)\}_{t\ge 0}$ is a generalized $Q$-Wiener process on a filtered probability space $(\Omega,\mathcal F,\{\mathcal F_t\}_{t\ge 0},\PP)$ and the nonlinearity $f$ is the Nemyskii operator of $F'$, where the energy functional $F$ is a polynomial of degree 4, i.e., $c_4\xi^4+c_3\xi^3+c_2\xi^2+c_1\xi+c_0$ with $c_i\in \R$, $i=0,\cdots,4$, $c_4>0$.
A typical example is $F=\frac 14(\xi^2-1)^2$, which is double well potential and corresponds to the Cahn--Hilliard--Cook equation. 
The stochastic Cahn--Hilliard equation %\eqref{spde} %which is a stochastic partial differential equation (SPDE) with non-globally Lipschitz continuous nonlinearity
has been the object of several mathematical studies dealing with well-posedness (see e.g. \cite{DD96,DDT04,EM91}) or numerical approximations (see e.g. \cite{CCZZ18,FKLL18,HJ14,KZ13,KLM11,LM11}).
In \cite{CCZZ18,KZ13,LM11}, the strong convergence rate of numerical schemes for the linearized stochastic Cahn--Hilliard equation is considered.
For the non-globally Lipschitz coefficient case, \cite{FKLL18,KLM11} show the strong convergence and obtain the convergence rate in a large subsample space of the finite element method and its implicitly full discretization for equation \eqref{spde} driven by spatial regular noise.
\cite{HJ14} studies both the exponential integrability and  the strong convergence rate of the spectral Galerkin method for one 
dimensional equation \eqref{spde} driven by the trace class noise.
To the best of our knowledge, there exists no result about %the strong convergence of numerical schemes for stochastic Cahn--Hilliard equation driven by space-time white noise, as well as about  
the strong convergence rates of the temporal discretization and full discretization for the stochastic Cahn--Hilliard equation driven by rough noise, like the space-time white noise.

The present work aims to make further contributions on the strong convergence rates of numerical discretization for stochastic Cahn--Hilliard equation \eqref{spde} driven by  space-time white noise.
More precisely, we present the optimal strong convergence rates of both the spatial spectral Galerkin method
\begin{align}\label{sge}
dX^N(t)+A(AX^N(t)+P^Nf(X^N(t)))dt&=P^NdW(t),\\\nonumber
X^N(0)&=P^NX_0,
\end{align}
and the accelerated implicit Euler full discretization
\begin{align*}
Y^N_{k+1}&=Y^N_k-A^2Y^N_{k+1}\delta t
-AP^Nf(Y^N_{k+1}+Z^N(t_{k+1}))\delta t,\\\nonumber
X^N_{k+1}&=Y^N_{k+1}+Z^N(t_{k+1}), \; k\le K-1
\end{align*} 
where $P^N,$  $N\in \N_+,$ is the spectral Galerkin projection, $Z^N$ is the spectral projection of the stochastic convolution $Z(t)=\int_0^t\exp(-A^2(t-s))dW(s)$ and $\delta t$ is the time stepsize such that $T=K\delta t$, $K\in \N^+.$

To obtain the strong convergence rate of the above numerical solutions, we use an interpolation approach, rather than study the strong convergence problem in the Hilbert space $H:=L^2(\mathcal{O})$ directly.
The main idea of the interpolation approach is firstly deducing the optimal regularity estimates of the exact and numerical solutions in strong sense. 
Then we investigate the strong convergence in a negative Sobolev space, %such as the negative Sobolev space $\HH^{-1}$ in the case of stochastic Cahn--Hilliard equation,
and obtain the optimal strong convergence error estimate in such space.
Combining the established optimal regularity estimate and error estimate in the negative Sobolev space, with the interpolation inequality of the Sobolev spaces, we are able to obtain the optimal strong convergence rate of numerical scheme.

For the strong convergence analysis of the spatial spectral Galerkin method, 
we first show the optimal spatial and temporal regularity estimates of the exact and numerical solutions. % to overcome the difficulty brought by both the non-globally monotone nonlinearity and the driven noise which is rougher than the trace class noise.
Inspired by the fact that the deterministic Cahn--Hilliard equation defines a gradient flow in $\HH^{-1}$ for the energy functional, we focus on the strong convergence problem in $\HH^{-1}$.
By using the equivalence  between  equation \eqref{spde} and the system which consists of a random PDE and a stochastic convolution $Z(t)$, we deduce the sharp strong convergence error estimate in $\HH^{-1}$.
Then based on the Sobolev interpolation inequality and the smoothing effect of the semigroup, we recover the optimal strong convergence rate of the spectral Galerkin method for equation \eqref{spde} driven by space-time white noise.
Let $(I-\mathbb L) X_0\in \HH^{\gamma}$ for some $\gamma \in (0,\frac 32)$, $N\in \N^+$ and  $p\ge1$, where $\mathbb L$ is the projection operator satisfying $\mathbb L v=|\mathcal O|^{-1}\int_{\mathcal O}vdx$, $v\in H$.
The numerical solution $X^N$ is shown to strongly converge to $X$ and satisfies 
\begin{align*}
\big\|X^N(t)-X(t)\big\|_{L^p(\Omega;{H})}
&\le C(X_0,T,p)\lambda_N^{-\frac \gamma 2}
\end{align*}
for a positive constant $C(X_0,T,p)$.

For the accelerated full discretization, we first prove the optimal regularity of the full discretization following the interpolation approach.
Then by introducing an auxiliary process $\widetilde Y^N_k$, $k\le K$, we divide the temporal error in $\HH^{-1}$ into two parts, $\|Y^N(t_k)-\widetilde Y^N_k\|_{\HH^{-1}}$ and $\|\widetilde Y^N_k- Y^N_k\|_{\HH^{-1}}$. Based on the monotonicity of 
$-f$ and the interpolation arguments, we obtain the strong convergence of the proposed full discretization, i.e., let $(I-\mathbb L) X_0\in \HH^{\gamma}$ for all $\gamma \in (0,\frac 32)$, $N\in \N^+$ and  $p\ge1$, then
the numerical solution $X^N_k$ is strongly convergent to $X$  and  satisfies 
\begin{align*}
\big\|X^N_k-X(t_k)\big\|_{L^p(\Omega;{H})}
&\le C(X_0,T,p)(\delta t^{(\frac \gamma 2)^-}+\lambda_N^{-\frac \gamma 2})
\end{align*}
for a positive constant $ C(X_0,T,p)$.
We remark that this approach is also available for deducing the strong convergence rates of numerical schemes for equation \eqref{spde} driven by general noise for equation \eqref{spde} with $d\le 3$ (see Section \ref{sec-oth}).

The outline of this paper is as follows. In the next section,
some preliminaries are listed.
Section \ref{sec-pri} is devoted to giving the useful regularity and a priori estimates of equation \eqref{spde}.
In Section \ref{sec-str}, we prove some a priori estimates of the numerical solutions, and use the interpolation approach to study the strong convergence rates of both the spectral Galerkin method and its accelerated implicit full discretization.
Some applications of the interpolation approach to
the cases $d\le 3$ and general noises are presented in Section \ref{sec-oth}.

\section{Stochastic Cahn--Hilliard equation with cylindrical Wiener process}
\label{sec-pri}
In this section, we present both the optimal spatial and temporal regularity estimates of the exact solution $X$ for one dimensional stochastic Cahn--Hilliard equation \eqref{spde} driven by space-time white noise , that is, $d=1$ and $Q=I$.
For the well-posedness of the exact solution, we refer to \cite{AKM16,DD96} and references therein.

To investigate the regularity estimates, we introduce some notations.
Given two separable Hilbert spaces $(\mathcal H, \|\cdot \|_{\mathcal H})$ and $(\widetilde  H,\|\cdot \|_{\widetilde H})$, 
$\LL(\mathcal H, \widetilde H)$ and $\LL_1(\mathcal H, \widetilde H)$ are the Banach space of all linear bounded operators 
and that of the nuclear operators from $\mathcal H$ to $\widetilde H$, respectively. 
The trace of an operator $\mathcal T\in \LL_1(\mathcal H)$
is $tr[\mathcal T]=\sum_{k\in \N^+}\<\mathcal Tf_k,f_k\>_{\mathcal H}$, where $\{f_k\}_{k\in \N^+}$ is any orthonormal basis of $\mathcal H$.
%In particular, if $\mathcal T\ge 0$, $tr[\mathcal T]=\|\mathcal T\|_{\mathcal L_1}$.
Denote by $\LL_2(\mathcal H,\widetilde H)$ the space 
of Hilbert--Schmidt operators from $\mathcal H$ into $\widetilde H$, equipped with the usual norm given by  $\|\cdot\|_{\LL_2(\mathcal H,\widetilde H)}=(\sum_{k\in \N^+}\|\cdot f_k\|^2_{\widetilde H})^{\frac{1}{2}}$.
\iffalse
The following useful property and inequality hold 
\begin{align}\label{tr}
&\|\mathcal S \mathcal T\|_{\mathcal L_2(\mathcal H,\widetilde H)}\le \|\mathcal S\|_{\mathcal L_2(\mathcal H,\widetilde H)}\|\mathcal T\|_{\mathcal L(\mathcal H)},\quad \mathcal T \in \mathcal L(\mathcal H), \;\;\mathcal S\in\mathcal L_2(\mathcal H,\widetilde H),\\\nonumber
&tr[\mathcal Q]=\|\mathcal Q^{\frac 12}\|^2_{L_2(\mathcal H)}=\|\mathcal T\|^2_{ \mathcal L_2(\widetilde H,\mathcal H)},\quad \mathcal Q=\mathcal T \mathcal T^{*},\;\; \mathcal T\in \mathcal L_2(\widetilde H,\mathcal H),
\end{align}
where $\mathcal T^*$ is the adjoint operator of $\mathcal T$.
\fi

Given a Banach space $(\mathcal E,\|\cdot\|_{\mathcal E})$, denote $\gamma( \mathcal H, \mathcal E)$ the space of $\gamma$-radonifying operators endowed with the norm
$\|\mathcal T\|_{\gamma(\mathcal  H, \mathcal E)}=(\widetilde \E\|\sum_{k\in\N^+}\gamma_k \mathcal Tf_k \|^2_{\mathcal E})^{\frac 12}$,
where $(\gamma_k)_{k\in\N^+}$ is a sequence of independent $\mathcal N(0,1)$-random variables on a
probability space $(\widetilde \Omega,\widetilde \FFF, \widetilde \PP)$.
For convenience,
let $\{e_k\}_{k\in \N}$ be an orthonormal basis of $H:=L^2(\OO)$ equipped with $\|\cdot\|=\|\cdot\|_{H}$ and $\<\cdot,\cdot\>=\<\cdot,\cdot\>_{H}$.
Denote $L^q:=L^q(\OO)$, $2\le q<\infty$ and $E:=\mathcal C(\OO)$ equipped with the usual norms.
For $H$-valued cylindrical 
Wiener process, we have the following Burkerholder inequality in $L^q,$ 
\begin{align}\label{Burk}
\left\|\sup_{t\in [0,T]}\Big\|\int_0^t \phi(r)d\widetilde W(r)\Big\|_{L^q}\right\|_{L^p(\Omega)}
&\le 
C_{p}\|\phi\|_{L^p(\Omega;  L^2([0,T]; \gamma(H;L^q))}\\\nonumber 
&\le  C_{p}\Big(\E\Big(\int_0^T\Big\|\sum_{k\in \N^+} (\phi(t) e_k)^2\Big\|_{L^{\frac q2}}dt\Big)^{\frac p2}\Big)^{\frac 1p}.
\end{align}

Denote $H^k:=H^k(\mathcal O)$ the standard Sobolev space and define $A:=-\Delta $ as the Neumann Laplacian operator with 
$$D(A)=\left\{v\in H^2(\mathcal O):\frac {\partial v}{\partial n}=0\;\; \text{on} \;\; \partial \mathcal O\right\}.$$
Let $\mathbb L$ be the projection $\mathbb L v:=|\mathcal O|^{-1}\int_{\mathcal O}vdx$, $v\in H$.
Let $\HH=(I-\mathbb L)H$.
It is known that $A$ is a positive definite, self-adjoint and unbounded linear operator on $\HH$.
By extending $A$ on $H$, $A$ has an orthonormal 
eigensystem $\{(\lambda_j,e_j)\}_{j\in \N}$ such that $0=\lambda_0<\lambda_1\le \cdots \le \lambda_j\le \cdots$ with $\lambda_j \sim j^{\frac 2d}$, $e_0=|\mathcal O|^{-\frac12}$.
Define 
$\HH^{\alpha}$, $\alpha\in \R$ as the space of the series $v:=\sum_{j\in\mathbb N_+}v_je_j$, $v_j\in \R$, such that
$\|v\|_{\HH^{\alpha}}:=(\sum_{j=1}^{\infty}\lambda_j^{\alpha}v_j^2)^{\frac 12}<\infty$. 
Equipped with the norm $\|\cdot\|_{\HH^{\alpha}}$ and corresponding inner product, the Hilbert space $\HH^{\alpha}$ equals $D(A^{\frac \alpha 2})$.
When $\bs=1,2$,  the norm $||\cdot||_{\mathbb L H\oplus \HH^{\bs}}$ of $\mathbb L H\oplus \HH^{\bs}$ is equivalent to the Sobolev norm $\|\cdot\|_{H^{\bs}}$, i.e., for $v\in \mathbb L H\oplus \HH^{\bs}$,
$$||v||_{\mathbb L H\oplus \HH^{\bs}}^2:=\|v\|^2_{\HH^\bs}+|\<v,e_0\>|^2=\|\nabla^{\bs}v\|^2+|\<v,e_0\>|^2\sim\|v\|_{H^{\bs}}.$$
 The following smoothing effect of the analytical semigroup $S(t)=e^{-t A^2},t>0$ (see e.g. \cite{EN00}), 
\begin{align}\label{smo-eff}
\|A^{\beta}S(t)v\|&\le Ct^{-\frac \beta 2}\|v\|, \;\beta>0, \; v\in \mathbb H,
\end{align}
and the contractivity property of $S(t)$ (see e.g. \cite[Appendix B]{PZ07}),
\begin{align}\label{con-pro}
\|S(t)v\|_{L^q}&\le Ct^{-\frac d{4}(\frac 1p-\frac 1q)} \|v\|_{L^p},\; 1\le p\le q<\infty ,\;  v\in L^p,\\\nonumber 
\|S(t)v\|_{E}&\le Ct^{-\frac d{4p}} \|v\|_{L^p}, v\in L^p,
\end{align}
will be used frequently.

Throughout this article, we denote the one-side Lipschitz coefficient of $f,$ by $L_f$ and give the following assumptions.
 In Sections 2 and 3, we assume that the driving noise corresponds to the space-time white noise, i.e., $d=1$, $Q=I$. Our approach is also available to the case that  
$Q$ commutes with $A$ and $\|A^{\frac {\gamma-2}2}Q^{\frac 12}\|_{\LL_2^0}<\infty$ for $\gamma\in (0,\frac 32)$. 
We consider the general space dimension case, i.e., $\mathcal O=[0,L]^d$, $d\le 3$ driven by additive noise satisfying $\|A^{\frac {\gamma-2}2}Q^{\frac 12}\|_{\LL_2^0}<\infty$ with $\gamma>0$  In Section 4.

For the space-time white noise case, it is known (see e.g. \cite{DD96}) that  the stochastic convolution $Z(t)$, as the solution of 
\begin{align*}
&dZ+A^2Zdt=dW(t),\; Z(0)=0,
\end{align*}
satisfies
\begin{align}\label{reg-con}
\E\Big[\sup_{t\in [0,T]}\|(I-\mathbb L)Z(t)\|_E^p\Big]+\E\Big[\sup_{t\in [0,T]}\|(I-\mathbb L)Z(t)\|_{\HH^{\gamma}}^p\Big]\le C(p,T).
\end{align}
for any $p\ge 1$ and $\gamma<\frac 32$, and 
\begin{align}\label{reg-con1}
\E\Big[\big\|Z(t)-Z(s)\big\|^p\Big]&\le C(T,p)(t-s)^{\frac {\gamma p} 4},
\end{align}
for $s\le t$.
Let $Y$ satisfy
\begin{align*}
&dY+(A^2Y+Af(Y+Z))dt=0,\; Y(0)=X_0.
\end{align*}
which yields that $X=Y+Z.$
Furthermore, based on the projection operator, $X$ has the decomposition
$X=\mathbb L X+(I-\mathbb L)X$,
where $\mathbb L X$ satisfies 
\begin{align*}
d\mathbb LX+A^2\mathbb LXdt
+A\mathbb L(f(X))dt=\mathbb LdW(t)
\end{align*}
and 
$(I-\mathbb L)X$ satisfies 
\begin{align*}
d(I-\mathbb L)X+A^2(I-\mathbb L)Xdt
+A(I-\mathbb L)(f(X))dt=(I-\mathbb L)dW(t).
\end{align*}

Based on the above decomposition, the a priori estimates of both $Z$ and $Y,$ we present the $p$th moment estimate of $\|X\|$ in the following lemma.

\begin{lm}\label{l2}
Let $X_0\in H$ and $p\ge1$.
There exists a unique mild solution  $X$  of equation \eqref{spde}  satisfying 
\begin{align}\label{p-mom}
\E\Big[\sup_{t\in [0,T]}\big\|X(t)\big\|^p\Big]
&\le C
\end{align}
for a positive constant $ C:=C(X_0,T,p)$.
\end{lm}
\begin{proof}
Since 
$(I-\mathbb L)X=(I-\mathbb L)Y+(I-\mathbb L)Z$, where $(I-\mathbb L)Y$ satisfies the following random PDE 
\begin{align}\label{vspde}
&d(I-\mathbb L)Y+A^2(I-\mathbb L)Ydt+A(I-\mathbb L)(f(Y+Z))dt=0,\\\nonumber 
&(I-\mathbb L)Y(0)=(I-\mathbb L)X(0),
\end{align}
and $(I-\mathbb L)Z$ satisfies  
\begin{align*}
d(I-\mathbb L)Z+A^2(I-\mathbb L)Zdt=(I-\mathbb L)dW(t),\; (I-\mathbb L)Z(0)=0,
\end{align*}
it suffices to estimate the term $\|(I-\mathbb L)Y\|$ due to the a priori estimate \eqref{reg-con} of $Z$.
Before that, we first give the estimate of $\|(I-\mathbb L)Y\|_{\HH^{-1}}$.
For any sufficiently small number $\epsilon>0$, it follows from the chain rule, the dissipative  property of $-f$, the H\"older and Young inequalities  
that  
\begin{align*}
&\quad\|(I-\mathbb L)Y(t)\|_{\HH^{-1}}^2\\
%&\le \|(I-\mathbb L)X_0\|_{\HH^{-1}}^2-2\int_0^t\|(I-\mathbb L)Y(s)\|_{\HH^1}^2ds\\
%&\quad-2\int_0^t\<f(Y(s)+Z(s)),(I-\mathbb L)Y(s)\>ds\\
&\le
\|(I-\mathbb L)X_0\|_{\HH^{-1}}^2
-(2-\epsilon)\int_0^t\|(I-\mathbb L)Y(s)\|_{\HH^1}^2ds\\
&\quad
+\int_0^t\|(I-\mathbb L)Y(s)\|_{\HH^{-1}}^2ds
-8(c_{4}-\epsilon)\int_{0}^t\|(I-\mathbb L)Y(s)\|_{L^{4}}^{4}ds\\
&\quad+C(\epsilon)\int_0^t\Big(1+\|\mathbb L Y(s)\|_{L^{4}}^{4}
+\|Z(s)\|_{L^{4}}^{4}\Big)ds.
\end{align*}
Then the Gronwall inequality leads that for $t\le T$,
\begin{align*}
\|(I-\mathbb L)Y(t)\|_{\HH^{-1}}^2
&\le
e^{T} \|(I-\mathbb L)X_0\|_{\HH^{-1}}^2
+C(\epsilon,T)\int_0^T\Big(1+\|\mathbb L Y(s)\|_{L^{4}}^{4}ds\\
&\qquad
+\|Z(s)\|_{L^{4}}^{4}\Big)ds,
\end{align*}
which implies that
\begin{align*}
&(2-\epsilon)\int_0^t\|(I-\mathbb L)Y(s)\|_{\HH^1}^2ds+8(c_4-\epsilon)\int_{0}^t\|(I-\mathbb L)Y(s)\|_{L^{4}}^{4}ds\\
&\le C(\epsilon, T)\Big(\|(I-\mathbb L)X_0\|_{\HH^{-1}}^2
+\int_0^T\left(1+\|\mathbb L Y(s)\|_{L^{4}}^{4}
+\|Z(s)\|_{L^{4}}^{4}\Big)ds\right).
\end{align*}
Due to the definition of $\mathbb L$ and 
\begin{align*}
d\mathbb LY+\mathbb LA^2Ydt
+\mathbb LAf(Y+Z)dt=0,\; \mathbb LY(0)=\mathbb L X_0,
\end{align*}
we have $\mathbb LY(t)=\mathbb L X_0$, which implies  $\|\mathbb LY(t)\|_{L^{4}}=\|\mathbb LX_0\|_{L^{4}}$. 
The a priori estimates of $Z$ and $\mathbb L Y$ yield the uniformly boundedness of  both the $p$th moment of $\int_0^t\|(I-\mathbb L)Y(s)\|_{\HH^1}^2ds$ and that of $\int_0^T\|(I-\mathbb L)Y(s)\|_{L^{4}}^{4}ds.$

Now we are in the position to give the a priori estimate of $\|(I-\mathbb L)Y\|$. By applying the chain rule and integration by parts,
we obtain 
\begin{align*}
&\quad\|(I-\mathbb L)Y(t)\|^2\\%&= \|(I-\mathbb L)X_0\|^2-2\int_0^t\|(-A)(I-\mathbb L)Y(s)\|^2ds\\
%&\quad-2\int_{0}^t\<\nabla f(Y(s)+Z(s)),\nabla(I-\mathbb L)Y(s)\>ds\\
&\le \|(I-\mathbb L)X_0\|^2
-(2-\epsilon)\int_0^t\|(-A)(I-\mathbb L)Y(s)\|^2ds\\
&\quad -(24c_{4}-\epsilon)\int_0^t\|Y(s)\nabla Y(s)\|^{2}ds +C(\epsilon)\int_0^t\| Z(s)\|_E^2\|Y(s)\|_{L^4}^4ds\\
&\qquad+
C(\epsilon)\int_0^t
(1+\|\nabla(I-\mathbb L)Y(s)\|^2+\|Z(s)\|_{L^6}^6)ds.
\end{align*}
By the equivalence of norms in $\HH^{1}$ and $H^1$ for the functions in $\HH^1$, we have 
\begin{align*}
&\|(I-\mathbb L)Y(t)\|^2+(2-\epsilon)\int_0^t\|(-A)(I-\mathbb L)Y(s)\|^2ds\\
&\le \|(I-\mathbb L)X_0\|^2
+C(\epsilon)\int_0^t
\big(1+\|(I-\mathbb L)Y(s)\|_{\HH^1}^2+\|Z(s)\|_{L^6}^6\\
&\qquad+\|Z(s)\|_E^2\|Y(s)\|_{L^4}^4\big)ds.
\end{align*}
Combining the uniform boundedness  of both the $p$th moment of $\int_0^T\|(I-\mathbb L)Y(s)\|_{\HH^1}^2ds$ and $\int_0^T\|(I-\mathbb L)Y(s)\|_{L^{4}}^{4}ds$ with the estimate $\E\Big[\sup\limits_{s\in[0,T]}\|Z(s)\|_{E}^{p}\Big]\le C(p,T)$ for $p\ge 1$, it follows that
\begin{align}\label{pri-u2}
\E\Big[\sup_{t\in [0,T]}\|(I-\mathbb L)Y(t)\|^{2p}+\Big(\int_0^T\|(-A)(I-\mathbb L)Y(s)\|^2ds\Big)^{p}\Big]
&\le C(X_0,T,p),
\end{align}
which, together with the a priori estimates of both $\|\mathbb LY\|$ and $\|Z\|$, completes the proof.
\end{proof}

Based on the a priori estimate of $\|X\|$, we give the following regularity estimate of $X$.
\begin{prop}\label{prop-spa}
Let $(I-\mathbb L) X_0\in \HH^{\gamma}$, $\gamma \in (0,\frac 32)$, $p\ge1$.
The unique mild solution  $X$  of equation \eqref{spde}  satisfies 
\begin{align}\label{reg}
\E\Big[\sup_{t\in [0,T]}\big\|(I-\mathbb L)X(t)\big\|_{\HH^{\gamma}}^p\Big]
&\le C
\end{align}
for a positive constant $ C:=C(X_0,T,p)$.
\end{prop}
\begin{proof}
Due to \eqref{reg-con},
it suffices to give a regularity estimate for $(I-\mathbb L)Y$.
Before that, we give the following estimate of $\|(I-\mathbb L)Y(t)\|_{L^6}$.
The Sobolev embedding theorem $L^q\hookrightarrow H^{\frac 12-\frac 1q},  2\le q<\infty$, the contractivity property of $e^{-A^2t}$ \eqref{con-pro}   and the Gagliardo--Nirenberg inequality yield that 
\begin{align*}
&\quad\|(I-\mathbb L)Y(t)\|_{L^6}\\
&\le \|e^{-A^2t}(I-\mathbb L)X_0\|_{L^6}+\int_0^t \|e^{-A^2(t-s)}A(1-\mathbb L)f(Y(s)+Z(s))\|_{L^6}ds\\
%&\le \min\{\|e^{-A^2t}(I-\mathbb L)X_0\|_{L^6},\|e^{-A^2t}A^{\frac 14+\epsilon}(I-\mathbb L)X_0\|\}\\
%&\quad+\int_0^t \|e^{-A^2(t-s)}A(1-\mathbb L)f(Y(s)+Z(s))\|_{L^6}ds\\
&\le Ct^{\min(-\frac 1{12}+\frac \gamma 4,0)}\|X_0\|_{\mathbb H^{\gamma}}
+C\int_0^t(t-s)^{-\frac 7{12}}
\Big(1+\|(I-\mathbb L)Y(s)\|_{L^{6}}^{3}\\
&\qquad+\|Z(s)\|_{L^{6}}^{3}+\|\mathbb L Y(s)\|_{L^{6}}^{3}\Big)ds\\
&\le Ct^{\min(-\frac 1{12}+\frac \gamma 4,0)}\|X_0\|_{\mathbb H^{\gamma}}
+C\int_0^t(t-s)^{-\frac 7{12}}
\Big(1+\|Z(s)\|_{L^{6}}^{3}+\|\mathbb L Y(s)\|_{L^{6}}^{3}
\\
&\qquad+
\|(-A)(I-\mathbb L)Y(s)\|^{\frac 12}\|(I-\mathbb L)Y(s)\|^{\frac 52}+
\|(I-\mathbb L)Y(s)\|\Big)ds.
\end{align*}
From the H\"older and Young inequalities, the a priori estimates of both $Z$ 
and $\mathbb L Y$, and the estimate \eqref{pri-u2}, 
it follows that for any $p\ge 1$,
\begin{align}\label{l6}
&\quad\|(I-\mathbb L)Y(t)\|_{L^p(\Omega;L^6)}\nonumber\\
&\le   Ct^{\min(-\frac 1{12}+\frac \gamma 4,0)}(\|X_0\|_{L^p(\Omega;\mathbb H^{\gamma})}+1)
+C(p)\Big(\E\Big[\Big(\int_0^T\|(-A)(I-\mathbb L)Y(s)\|^2ds\Big)^{\frac p2}\Big]\Big)^{\frac 1p}
\nonumber\\
&\quad+C(\int_0^T(t-s)^{-\frac 7{9}}ds)^{\frac {3}2}\Big(\E\Big[\sup_{s\in[0,T]}\|(I-\mathbb L)Y(s)\|^{\frac {10p}3}\Big]\Big)^{\frac 1p}\nonumber \\
&\quad+C\int_0^T(t-s)^{-\frac 7{12}}ds
\Big(\E\Big[\sup_{s\in [0,T]}\|(I-\mathbb L)Y(s)\|^p\Big]\Big)^{\frac 1p}\nonumber\\ 
&\le C(T,X_0,p)(1+ t^{\min(-\frac 1{12}+\frac \gamma 4,0)}).
\end{align}
From the mild form of $(I-\mathbb L)Y(t)$ of equation \eqref{vspde} and the smoothing effect of $S(t)$ \eqref{smo-eff}, it follows that 
\begin{align*}
&\quad\|(I-\mathbb L)Y(t)\|_{\mathbb \HH^{\gamma}}\\
&\le \|e^{-A^2t}(I-\mathbb L)X_0\|_{\HH^{\gamma}}
+\int_0^t\big\|e^{-A^2(t-s)}A(I-\mathbb L)f(Y(s)+Z(s))\big\|_{\HH^{\gamma}}ds\\
&\le C\|X_0\|_{\HH^{\gamma}}
+C\int_0^t(t-s)^{-\frac 12}\big\|e^{-\frac 12A^2(t-s)}f(Y(s)+Z(s))\big\|_{\HH^{\gamma}}ds\\
&\le C\|X_0\|_{\HH^{\gamma}}
+C\int_0^t(t-s)^{-\frac 12-\frac {\gamma}4}\big(1+\|(I-\mathbb L)Y(s)\|_{L^6}^3+\|\mathbb LY(s)\|^3_{L^6}
\\
&\qquad+\|Z(s)\|_{L^6}^3\big)ds.
\end{align*}
It can be directly verified that $\int_0^t(t-s)^{-\frac 12-\frac\gamma 4}s^{\min(-\frac 14+\frac {3\gamma}4,0)}ds\le C(T).$
By taking $L^p(\Omega)$ norm and making use of the a prior estimates of $\|\mathbb LY(s)\|_{L^6}$, $\|(I-\mathbb L)Y(s)\|_{L^6}$ and $\|Z(s)\|_{E}$, we obtain \eqref{reg}.
\iffalse
For $\gamma\in[1,\frac 32),$ based on the Sobolev embedding theorem, we have
\begin{align*}
\|e^{-A^2t}(I-\mathbb L)X_0\|_{L^6}\leq \|X_0\|_{\mathbb H^\gamma}.
\end{align*}
Repeating the above arguments, we finish the proof.
\fi

\end{proof}

\begin{rk}\label{reg-y}
Let the  conditions of Proposition \ref{prop-spa} hold.
If $\gamma\ge \frac 13$, we have that 
$\E\Big[\sup\limits_{t\in[0,T]}\big\|(I-\mathbb L)Y(t)\big\|_{L^6}^p\Big] \le C(T,X_0,p)$ due to the Sobolev embedding theorem. 
If in addition assume that  $(I-\mathbb L)X_0 \in \HH^{\beta}$ for any $\beta<2$, we obtain the higher regularity estimate of $Y$, i.e.,
\begin{align*}
\E\Big[\sup_{t\in[0,T]}\big\|(I-\mathbb L)Y(t)\big\|_{\HH^{\beta}}^p\Big]
&\le C(T,X_0,p).
\end{align*}
\end{rk}
Similar to the proof of \eqref{prop-spa}, we have the following a priori estimate of the solution in $L^p(\Omega;E)$ norm.

\begin{prop}\label{pri-xE}
Let $(I-\mathbb L) X_0\in E$ and  $p\ge1$.
The unique mild solution  $X$  of equation \eqref{spde}  satisfies 
\begin{align}\label{pri-E}
\E\Big[\sup_{t\in [0,T]}\big\|(I-\mathbb L)X(t)\big\|_{E}^p\Big]
&\le C
\end{align}
for a positive constant $C:=C(X_0,T,p)$.
\end{prop}

\begin{prop}\label{prop-tm}
Let $(I-\mathbb L) X_0\in \HH^{\gamma}$, $\gamma \in (0,\frac 32)$, $p\ge1$.
The unique mild solution $X$ of equation \eqref{spde}  satisfies 
\begin{align}\label{reg-tm}
\E\Big[\big\|X(t)-X(s)\big\|^p\Big]
&\le C(t-s)^{\frac {\gamma p} 4}
\end{align}
for a positive constant $ C:=C(X_0,T,p)$ and $0\le s\le t\le T$.
\end{prop}
\begin{proof}
Due to the continuity \eqref{reg-con1} of $Z$  and the fact that $\mathbb L Y(t)=\mathbb L X_0$, it suffices to estimate $\|(I-\mathbb L)(Y(t)-Y(s))\|$. 
 Based on the mild form of $(I-\mathbb L)(Y(t)-Y(s))$ and  \eqref{smo-eff},  
we obtain that for some $\beta<2$,
\begin{align*}
&\|(I-\mathbb L)(Y(t)-Y(s))\|\\
&\le \left\|(I-\mathbb L)e^{-A^2s}(e^{-A^2(t-s)}-I)X_0\right\|
\\
&\quad+\int_0^s\Big\|(e^{-A^2(t-r)}-e^{-A^2(s-r)})A(I-\mathbb L)f(Y(r)+Z(r))\Big\|dr
\\
&\quad+\int_s^t\Big\|e^{-A^2(t-r)}Af(Y(r)+Z(r))\Big\|dr\\
&\le C\|(I-\mathbb L)X_0\|_{\HH^{\gamma}}(t-s)^{\frac \gamma4}
+\int_s^t (t-r)^{-\frac 12}\|f(Y(r)+Z(r))\|dr\\
&\quad+\int_0^s(s-r)^{-\frac 12-\frac \beta 4}\|A^{-\frac \beta 2}(e^{-A^2(t-s)}-I)\|\|f(Y(r)+Z(r))\|dr.
\end{align*}
The estimate 
\eqref{l6}, the H\"older inequality and the Minkowski inequality yield that 
for  $\gamma<\frac 13$ and a sufficiently small $\epsilon>0$,
\begin{align*}
&\Big\|\int_s^t (t-r)^{-\frac 12}\|f(Y(r)+Z(r))\|dr\Big\|_{L^p(\Omega)}\\
&\le C(\epsilon)\left(\int_s^t (t-r)^{-\frac 12(\frac {4-\epsilon}{3\gamma+3-\epsilon})}ds\right)^{\left(\frac {3\gamma+3-\epsilon}{4-\epsilon} \right)}\\
&= C(t-s)^{\frac {3\gamma+1-\frac {\epsilon}2}{4-\epsilon}},%\leq  C(t-s)^{\frac \gamma 4}.
\end{align*}
and for  $\gamma\geq\frac 13$, 
\begin{align*}
\Big\|\int_s^t (t-r)^{-\frac 12}\|f(Y(r)+Z(r))\|dr\Big\|_{L^p(\Omega)}
&\le C \int_s^t (t-r)^{-\frac 12}dr\leq C(t-s)^{\frac 12}.
\end{align*}
By using the Minkovskii inequality and \eqref{l6}, and taking $\beta<\min(2,3\gamma+1)$, we have that  
\begin{align*}
&\Big\|\int_0^s(s-r)^{-\frac 12-\frac \beta 4}\|A^{-\frac \beta 2}(e^{-A^2(t-s)}-I)\|\|f(Y(r)+Z(r))\|dr\Big\|_{L^p(\Omega)}\\
&\le C(t-s)^{\frac \beta 4}\int_0^s(s-r)^{-\frac 12-\frac \beta 4}r^{\min(-\frac 14+\frac {3\gamma} 4,0)}dr\\
&\le C(t-s)^\frac \beta 4.%^{\min (\frac \beta 4,\frac \gamma 4)}.
\end{align*}
Combining all the above estimates, we complete the proof.
\end{proof}

\begin{rk}
Let the  conditions of Proposition \ref{prop-tm} hold. 
If in addition $(I-\mathbb L)X_0 \in \HH^{\beta}$,
for any $\beta<2$, we can obtain the higher temporal  regularity of $Y$, i.e.,
\begin{align*}
\E\Big[\|Y(t)-Y(s)\|^p\Big]
&\le C(t-s)^{\frac {\beta p} 4}
\end{align*}
for a positive constant $ C:=C(X_0,T,p).$
\end{rk}

\section{Strong convergence rate of full discretization}
\label{sec-str}

In the section, we consider the semi-discretization and full discretization of  \eqref{spde}.
In contrast to the Lipschitz continuous case, both the convergence and the convergence rate of the numerical approximation for stochastic partial differential equation with non-globally Lipschitz continuous nonlinearity, become more involved recently (see e.g., \cite{BJ16,BJ13,BCH18,CHP12,CHP16,CH17,CH18,CHL16b,CHLZ17,CHS18,FKLL18,HJS18,KL18}) and are far from being well-understood.
It motivates us to give the strong convergence rate of the proposed full discretization for \eqref{spde}.

%After giving the optimal spatial and temporal regularity of $X$,  we first show the semi-discretization and full discretization of equation \eqref{spde} in this section.
%

Denote $P^N$  the spectral Galerkin projection into the linear space  spanned by the first $(N+1)$ eigenvectors  $\{e_0,e_1,\cdots,e_N\}$.
Then the spatial semi-discretization is 
\begin{align}\label{semi}
dX^N+A\left(AX^N+P^N(f(X^N))\right)dt=&P^NdW,\; t\in (0,T];\; \\
X^N(0)=&P^NX_0.\nonumber
\end{align}
Let $\delta t$ be the time stepsize such that $T=K\delta t$ for some $K\in \N^+$. By using the accelerated idea, we propose  
the full discrete numerical scheme 
\begin{align}\label{full0}
Y^N_{k+1}&=Y^N_k-A^2Y^N_{k+1}\delta t
-AP^Nf(Y^N_{k+1}+Z^N(t_{k+1}))\delta t,\\\nonumber
X^N_{k+1}&=Y^N_{k+1}+Z^N(t_{k+1}), \; k\le K-1
\end{align} 
with the initial data $Y^N_0=X^N(0), Z^N(0)=0$.
Then we have the mild form of $Y^N_k$,
\begin{align}\label{full}
Y^N_{k+1}=T_{\delta t} Y^N_{k}
-\delta tT_{\delta t}A
P^N(f(Y^N_{k+1}+Z^N(t_{k+1}))), \; k\le K-1,
\end{align}
where $T_{\delta t}=(I+A^2\delta t)^{-1}$.

To analyze the strong convergence rate of the proposed spectral Galerkin method, we introduce an interpolation approach, which is also used to study the strong convergence rate of the full discretization.

\subsection{Strong convergence rate 
of spectral Galerkin method}

The main idea of deducing the optimal convergence rates of numerical schemes lies on interpolation arguments.
Namely, we first need to obtain the optimal regularity of the exact and numerical solutions in an interpolation space $E_{\theta_1}, \theta_1\in \R$.
Then we need the optimal convergence rate of the proposed numerical scheme in another  interpolation space $E_{\theta_0}$, $\theta_0\in \R$, $\theta_0< \theta_1$. By the smoothing effect of the semigroup, we deduce the optimal convergence rate of the numerical scheme in any interpolation space between $E_{\theta_0}$ and $E_{\theta_1}$.

In our case, the interpolation space $E_{\theta_1}$
is the Sobolev space $\HH^{\gamma}$ for some $\gamma<\frac 32$ and $E_{\theta_0}$ is the space 
$\HH^{-1}$. Notice that $X^N=\mathbb LX^N+(I-\mathbb L)X^N$, $X^N=Y^N+Z^N$, 
where $Y^N$ and $Z^N$ satisfy
\begin{align*}
&dY^N+P^N(A^2Y^N+Af(Y^N+Z^N))dt=0,\; Y^N(0)=X^N(0),\\
&dZ^N+A^2Z^Ndt=P^NdW(t),\; Z^N(0)=0.
\end{align*}
It is obvious that the error 
$\|X^N-X\|$ can be split as 
\begin{align*}
\|X^N-X\|
&\le \|\mathbb L(Y^N-Y)\|
+\|(I-\mathbb L)(Y^N-Y)\|+\|Z^N-Z\|.
\end{align*}
The first term is $0$ due to the definition of $\mathbb L$ and $P^N$, and the last term is controlled directly by 
\begin{align*}
\E\Big[\|Z^N-Z\|^{p}\Big]\le C(T,p)\lambda_{N}^{-\frac {\gamma p }2}
\end{align*}
for $\gamma<\frac 32$.
 For the term $\|(I-\mathbb L)(Y^N-Y)\|$, it confronts at many troubles to directly estimate the strong convergence rate of numerical schemes for equation \eqref{spde} due to the nonlinear term $Af(X)$.
One of the main difficulties lies on the loss of the Gronwall inequality to deduce the convergence rate. 
Another difficulty is the lack of the regularity property of the exact solution due to the space-time white noise.
%If the driven noise has more regularity in space, such as the trace class noise, one can use the exponential integrability of numerical and exact solutions to deduce the strong  convergence rate.
To overcome these difficulties, we use the interpolation approach to deal with the term 
$\|(I-\mathbb L)(Y^N-Y)\|$.
The first step is the following optimal regularity of $X^N$, which is obtained by the similar arguments in the proofs of Propositions \ref{prop-spa} and \ref{prop-tm}.

\begin{lm}\label{pri-yn}
Let $(I-\mathbb L) X_0\in \HH^{\gamma}$, $\gamma \in (0,\frac 32)$, $p\ge1$.
Then $Y^N$  satisfies 
\begin{align*}
\E\Big[\sup_{t\in [0,T]}\big\|(I-\mathbb L)Y^N(t)\big\|_{\HH^{\gamma}}^p\Big]
&\le C(X_0,T,p),
\end{align*}
and 
\begin{align*}
\E\Big[\big\|Y^N(t)-Y^N(s)\big\|^p\Big]
&\le C(X_0,T,p)(t-s)^{\frac {\gamma p} 4}
\end{align*}
for a positive constant $ C(X_0,T,p)$.
\end{lm}

\begin{cor}\label{pri-xn}
Let $(I-\mathbb L) X_0\in \HH^{\gamma}$, $\gamma \in (0,\frac 32)$, $p\ge1$.
The unique mild solution $X^N$ of equation \eqref{semi}  satisfies 
\begin{align*}
\E\Big[\sup_{t\in [0,T]}\big\|(I-\mathbb L)X^N(t)\big\|_{\HH^{\gamma}}^p\Big]
&\le C(X_0,T,p),
\end{align*}
and 
\begin{align*}
\E\Big[\big\|X^N(t)-X^N(s)\big\|^p\Big]
&\le C(X_0,T,p)(t-s)^{\frac {\gamma p} 4}
\end{align*}
for a positive constant $ C(X_0,T,p)$.
\end{cor}

\begin{lm}\label{nabla}
Let $(I-\mathbb L) X_0\in \HH^{\gamma}$, $\gamma \in (0,\frac 32)$, $p\ge1$.
Then $X^N$ satisfies that for any small $\epsilon>0$
\begin{align*}
\big\|\nabla (I-\mathbb L)X^N(t)\big\|_{L^p(\Omega;E)}
&\le C(X_0,T,p)(1+t^{\min(-\frac 38-\epsilon+\frac \gamma 4,\frac 12-2\epsilon+\frac {3\gamma} 4)}).
\end{align*}
\end{lm}
\begin{proof}

By using the factorization method, it is not difficult to obtain that 
$\big\|\nabla (I-\mathbb L)Z^N(t)\big\|_{L^p(\Omega;E)}\le C(X_0,T,p).$
It only suffices to estimate $\nabla (I-\mathbb L)Y^N(t)$.
From the property of $e^{-tA^2}$, it follows that  
for any small $\epsilon>0$, 
\begin{align*}
&\|\nabla (I-\mathbb L)Y^N(t)\|_E\\
&
\le 
\|\nabla e^{-A^2t}(I-\mathbb L)Y^N(0)\|_E
+
\int_{0}^t \|\nabla e^{-A^2(t-s)}A(I-\mathbb L)f(Y^N(s)+Z^N(s))\|_Eds\\
&\le 
Ct^{-\frac 38-\epsilon+\frac \gamma 4}\|(I-\mathbb L)Y^N(0)\|_{\HH^{\gamma}}\\
&\quad+C\int_0^t(t-s)^{-\frac 12-\frac 38-\epsilon}\|f(Y^N(s)+Z^N(s))\|ds.
\end{align*}
Based on a priori estimates of $Y^N$ and $Z^N$, we have
\begin{align*}
\|\nabla (I-\mathbb L)Y^N(t)\|_{L^p(\Omega;E)}&\le 
Ct^{-\frac 38-\epsilon+\frac \gamma 4}+C\int_0^t(t-s)^{-\frac 12-\frac 38-\epsilon}(1+s^{\min(-\frac 14+\frac{3\gamma} 4,0)})ds\\
&\leq %Ct^{-\frac 38-\epsilon+\frac \gamma 4}+Ct^{\frac 18-\epsilon}=
Ct^{-\frac 38-\epsilon+\frac \gamma 4}(1+t^{\min(-\frac 18-\epsilon+\frac {3\gamma} 4,0)})\\
&\leq Ct^{\min(-\frac 38-\epsilon+\frac \gamma 4,-\frac 12-2\epsilon+\frac {3\gamma} 4)}.
\end{align*}

\end{proof}

\begin{rk}\label{rk-est-xn}
Let $(I-\mathbb L) X_0\in \HH^{\gamma}$, $\gamma \in (0,\frac 32)$, $p\ge 1$. 
Then $X^N$ satisfies that for any small $\epsilon>0$,
\begin{align}\label{rk-est-xn-e}
\big\| (I-\mathbb L)X^N(t)\big\|_{L^p(\Omega;E)}
&\le C(X_0,T,p)(1+t^{\min(-\frac 18-\epsilon+\frac \gamma 4,0)}).
\end{align}
If $(I-\mathbb L) X_0\in \HH^{\beta}$, $\beta>\frac 32$, we have 
\begin{align*}
\|\nabla (I-\mathbb L)Y^N(t)\|_{L^p(\Omega;C([0,T];E))}
&\le C(T,X_0,p)
\end{align*}
for $p\ge1$.
\end{rk}

The second step of the interpolation approach is 
proving the optimal strong convergence rate of the spectral Galerkin method in $\HH^{-1}$. To this end, we introduce the following useful results. 

\begin{lm}\label{sob}
Let $g: L^4\to H$ be the Nemyskii operator of a polynomial of second degree.  
Then it holds that for any $\beta\in (0,1)$ that 
\begin{align*}
\|(I-\mathbb L)g(x)y\|_{\HH^{-1}}\le C\big(1+\|x\|_E^2+\|(I-\mathbb L)x\|_{\HH^{\beta}}^2\big)\|y\|_{\HH^{-\beta}},
\end{align*}
where $ x\in E, (I-\mathbb L)x\in \HH^{\beta}$ and $y\in \HH.$
\end{lm}

Its proof can be found in the Appendix.

\begin{cor}\label{rk-sob}
Under the conditions of Lemma \ref{sob}, for any $\beta\in [1,\min(\eta,2))$ and $\eta>\frac 32,$ we have 
\begin{align*}
\|(I-\mathbb L)g(x)y\|_{\HH^{-\eta}}\le C\big(1+\|x\|_E^2+\|\nabla x\|_E^2+\|(I-\mathbb L)x\|_{\HH^{\beta}}^2\big)\|y\|_{\HH^{-\beta}},
\end{align*}
where $ x, \nabla x\in E, (I-\mathbb L)x\in \HH^{\beta}$ and $y\in \HH.$ 
\end{cor}

\begin{rk}\label{rk-sob1}
Similar to Corollary \ref{rk-sob}, we have  that for any $\beta\in [2,\min(\eta,3)]$ and $\eta>\frac 52$,
\begin{align*}
\|(I-\mathbb L)g(x)y\|_{\HH^{-\eta}}\le C\big(1+\|x\|_{W^{2,\infty}}^2+\|(I-\mathbb L)x\|_{\HH^{\beta}}^2\big)\|y\|_{\HH^{-\beta}},
\end{align*}
where $ x\in W^{2,\infty}, (I-\mathbb L)x\in \HH^{\beta}$ and $y\in \HH.$
We also have similar estimations in the cases  $d=2,3$,
where $\eta>\bs+\frac d2$, $\bs=1$ or $2$,
$\beta\in [\bs,\min(\eta,1+\bs))$.
\end{rk}

\begin{prop}\label{tm-strong}
Let $(I-\mathbb L) X_0\in \HH^{\gamma}$, $\gamma \in (0,\frac 32)$, $N\in \N^+$ and  $p\ge1$.
There exists a positive constant $ C:=C(X_0,T,p,\gamma)$ such that for any $\beta\in (0,1)$,
\begin{align}\label{strong}
\big\|Y^N(t)-Y(t)\big\|_{L^p(\Omega;\HH^{-1})}
&\le C\lambda_N^{-\frac \beta 2-\frac \gamma 2}.
\end{align}
\end{prop}
\begin{proof}
From the property of 
the Galerkin projection and the triangle inequality, it follows that  
\begin{align*}
\big\|Y^N-Y\big\|_{\HH^{-1}}
 &\le \big\|Y^N-P^NY\big\|_{\HH^{-1}}+\big\|(I-P^N)Y\big\|_{\HH^{-1}}\\
&\le \big\|Y^N-P^NY\big\|_{\HH^{-1}}+
C\lambda_N^{-\frac 12-\frac \gamma2}\|Y\|_{\HH^{\gamma}}.
\end{align*}
Thus it suffices to estimate the term $\big\|Y^N-P^NY\big\|_{\HH^{-1}}.$
By the chain rule, the integration by parts, 
the monotonicity of $-f$, the interpolation inequality for  Sobolev space (see e.g. \cite{Tar07}), and the Young inequality, we have 
\begin{align*}
&\|Y^N-P^NY\|^2_{\HH^{-1}}\\
&=-2\int_0^t\<\nabla(Y^N-P^NY),\nabla (Y^N-P^NY)\>ds\\
&\quad
-2\int_0^t\<f(Y^N+Z^N)-f(Y+Z),Y^N-P^NY\>ds\\
&\le -2\int_0^t\|\nabla (Y^N-P^NY)\|^2ds
-2\int_0^t\<\int_0^1f'(\theta Y^N+\theta Z^N+(1-\theta)Y+(1-\theta)Z)d \theta \\
&\qquad(Y^N-P^NY-(I-P^N)Y-(I-P^N)Z),Y^N-P^NY\>ds\\
&\le -2\int_0^t\|\nabla (Y^N-P^NY)\|^2ds
+2L_{f}\int_0^t\| Y^N-P^NY\|^2ds
\\
&\quad
+2\int_0^t\<A^{-\frac 12}(I-\mathbb L)\Big(\int_0^1f'(\theta (Y^N+ Z^N)+(1-\theta)(Y+Z))d \theta \\
&\qquad((I-P^N)Y+(I-P^N)Z)\Big), A^{\frac 12}(Y^N-P^NY)\>ds
\\
&\le -(2-\epsilon)\int_0^t\|\nabla (Y^N-P^NY)\|^2ds
+C(\epsilon) \int_0^t\|Y^N-P^NY\|_{\HH^{-1}}^2ds\\
&\quad+C(\epsilon)\int_0^t\Big\|(I-\mathbb L)\Big(\int_0^1f'(\theta (Y^N+ Z^N)+(1-\theta)(Y+Z))d \theta\\
&\qquad  
 ((I-P^N)Y+(I-P^N)Z)\Big)\Big\|_{\HH^{-1}}^2ds.
\end{align*}
It follows from Lemma \ref{sob} that for any $\beta\in (0,1)$,
\begin{align*}
&\|Y^N(t)-P^NY(t)\|^2_{\HH^{-1}}\\
&\le C\int_0^t\Big\|\int_0^1(I-\mathbb L)\Big(f'(\theta (Y^N+ Z^N)
+(1-\theta)(Y+Z)) \\
&\qquad((I-P^N)Y+(I-P^N)Z)d \theta 
\Big)\Big\|_{\HH^{-1}}^2ds\\
&\le 
C\int_0^t\Big(1+\|(I-\mathbb L)X^N\|_{\HH^{\beta}}^4+\|(I-\mathbb L)X\|_{\HH^{\beta}}^4\\
 &\qquad +\|X^N\|_E^4+\|X\|_E^4\Big)\|(I-P^N)Y+(I-P^N)Z\|_{\HH^{-\beta}}^2ds.
\end{align*}
Then from taking $L^{p}(\Omega)$ norm, the a priori estimates in Proposition \ref{prop-spa} and Corollary \ref{pri-xn},  \eqref{rk-est-xn-e}, and the H\"older inequality, it follows that for any $\beta\in (0,1)$,
\begin{align*}
&\|Y^N(t)-P^NY(t)\|_{L^{2p}(\Omega;\HH^{-1})}^2\\
&\le C(T)\Big(1+\int_0^T\big\|\|(I-\mathbb L)X^N\|_{\HH^{\beta}}^4+\|(I-\mathbb L)X\|_{\HH^{\beta}}^4+\|X^N\|_E^4\\
&\qquad +\|X\|_E^4\big\|_{L^{2p}(\Omega;\R)}ds\Big)
\sup_{s\in [0,T]}\|(I-P^N)Y(s)+(I-P^N)Z(s)\|_{L^{4p}(\Omega;\HH^{-\beta})}^2
\\
&\le C(T,X_0,p)\Big(\sup_{s\in [0,T]}\|A^{-\frac \beta 2}(I-P^N)Y(s)\|_{L^{4p}(\Omega;\HH)}^2\\
&\quad\quad+\sup_{s\in [0,T]}\|A^{-\frac \beta 2}(I-P^N)Z(s)\|_{L^{4p}(\Omega;\HH)}^2\Big)\\
&\le C(T,X_0,p)\lambda_N^{-\beta -\gamma }.
\end{align*}
Taking square root on both sides and using H\"older inequality, we complete the proof.
\end{proof}

Based on Lemma \ref{pri-yn} and Proposition \ref{tm-strong},  we deduce the optimal strong convergence rate of the proposed semi-discretization. 

\begin{tm}\label{tm-strong0}
Let $(I-\mathbb L) X_0\in \HH^{\gamma}$, $\gamma \in (0,\frac 32)$, $N\in \N^+$ and  $p\ge1$.
The numerical solution $X^N$ is strongly convergent to $X$  and  satisfies 
\begin{align}\label{strong0}
\big\|X^N(t)-X(t)\big\|_{L^p(\Omega;{H})}
&\le C\lambda_N^{-\frac \gamma 2}
\end{align}
for a positive constant $C=: C(X_0,T,p)$ and any sufficiently small $\epsilon>0$.
\end{tm}
\begin{proof}
From the triangle inequality, it follows that 
\begin{align*}
\big\|X^N(t)-X(t)\big\|_{L^p(\Omega;{H})}
&\le \big\|Y^N(t)-Y(t)\big\|_{L^p(\Omega;{H})}
+\big\|Z^N(t)-Z(t)\big\|_{L^p(\Omega;{H})}\\
&=\big\|Y^N(t)-Y(t)\big\|_{L^p(\Omega;\HH)}
+\big\|Z^N(t)-Z(t)\big\|_{L^p(\Omega;\HH)}\\
&\le \big\|Y^N(t)-Y(t)\big\|_{L^p(\Omega;\HH)}
+C(t,p)\lambda_N^{-\frac \gamma 2}.
\end{align*}
It suffices to estimate $\big\|Y^N(t)-Y(t)\big\|_{L^p(\Omega;\HH)}$.
From the mild form of $Y^N$ and $Y$, \eqref{smo-eff}, \eqref{rk-est-xn-e} and Lemma \ref{sob}, we have that for any $\beta\in (0,1)$ 
\begin{align*}
&\big\|Y^N(t)-Y(t)\big\|_{L^p(\Omega;\HH)}\\
&\le \|S(t)(I-P^N)Y_0\|_{L^p(\Omega;\HH)}+\int_0^t\|S(t-s)A(I-P^N)f(Y^N+Z^N)\|_{L^p(\Omega;\HH)}ds\\
&\quad+\int_0^t\|S(t-s)AA^{\frac 12}(f(Y^N+Z^N)-f(Y+Z))\|_{L^p(\Omega;\HH^{-1})}ds
\\
&\le C\lambda_N^{-\frac \gamma 2}
+C\int_0^t\|S(t-s)A^{\frac \gamma 2}AA^{-\frac \gamma 2}(I-P^N)f(Y^N+Z^N)\|_{L^p(\Omega;\HH)}ds\\
&\quad+C\int_0^t(t-s)^{-\frac 34}\big\|(1+\|X^N\|_E^2+\|X\|_E^2+\|X^N\|_{\HH^{\beta}}^2+\|X\|_{\HH^{\beta}}^2)\big\|_{L^{2p}(\Omega;\R)}\\
&\quad\times \left(\|Y^N-Y\|_{L^{2p}(\Omega;\HH^{-\beta})}+\|Z^N-Z\|_{L^{2p}(\Omega;\HH^{-\beta})}\right)ds.
\end{align*}
It can be verified that 
\begin{align*}
&\int_0^t\|S(t-s)A^{\frac \gamma 2}AA^{-\frac \gamma 2}(I-P^N)f(Y^N+Z^N)\|_{L^p(\Omega;\HH)}ds\\
&\leq \lambda_N^{-\frac \gamma 2}\int_0^t(t-s)^{-\frac 12-\frac \gamma 4}(1+s^{\min(-\frac 14+\frac {3\gamma}4,0)})ds\leq C(T)\lambda_N^{-\frac \gamma 2}.
\end{align*}
By further using the a priori estimate of $X$ and $X^N$, we have
\begin{align*}
&\int_0^t(t-s)^{-\frac 34}\big\|(1+\|X^N\|_E^2+\|X\|_E^2+\|X^N\|_{\HH^{\beta}}^2+\|X\|_{\HH^{\beta}}^2)\big\|_{L^{2p}(\Omega;\R)}\\
&\quad\times \left(\|Y^N-Y\|_{L^{2p}(\Omega;\HH^{-\beta})}+\|Z^N-Z\|_{L^{2p}(\Omega;\HH^{-\beta})}\right)ds\\
\leq& \int_0^t(t-s)^{-\frac 34}\big(1+\|X^N\|_{L^{4p}(\Omega;E)}^2+\|X\|_{L^{4p}(\Omega;E)}^2+\|X^N\|_{L^{4p}(\Omega;\mathbb H^\beta)}^2+\|X\|_{L^{4p}(\Omega;\mathbb H^\beta)}^2\big)\\
&\quad\times \left(\|Y^N-Y\|_{L^{2p}(\Omega;\HH^{-\beta})}+\|Z^N-Z\|_{L^{2p}(\Omega;\HH^{-\beta})}\right)ds\\
\leq& C\int_0^t(t-s)^{-\frac 34}\big(1+s^{\min(-\frac 14-\epsilon+\frac \gamma 2,0)})\left(\|Y^N-Y\|_{L^{2p}(\Omega;\HH^{-\beta})}+C\lambda_N^{-\frac \beta 2-\frac\gamma 2}\right)ds\\
\leq& C\lambda_N^{-\frac \beta 2-\frac\gamma 2}+\int_0^t(t-s)^{-\frac 34}\big(1+s^{\min(-\frac 14-\epsilon+\frac \gamma 2,0)})\|Y^N-Y\|_{L^{2p}(\Omega;\HH^{-\beta})}ds.
\end{align*}
Let $\theta=\frac {(1-\beta)}{\gamma+1}.$
The interpolation and H\"older inequalities lead to 
\begin{align*}
\|Y^N-Y\|_{L^{2p}(\Omega;\HH^{-\beta})}&\leq \|Y^N-Y\|_{L^{p}(\Omega;\HH^{-1})}^{1-\theta}\|Y^N-Y\|_{L^{p}(\Omega;\HH^{\gamma})}^\theta.
\end{align*}
This means that
\begin{align*}
&\big\|Y^N(t)-Y(t)\big\|_{L^p(\Omega;\HH)}\\
\leq& C\lambda_N^{-\frac \gamma 2}
+C
\int_0^t(t-s)^{-\frac 34}\big(1+s^{\min(-\frac 14-\epsilon+\frac \gamma 2,0)})\|Y^N-Y\|_{L^{p}(\Omega;\HH^{-1})}^{1-\theta}ds\\
\leq &C\lambda_N^{-\frac \gamma 2}+C(\lambda_N^{-\frac \beta 2-\frac\gamma 2})^{1-\theta}\le C\lambda_N^{-\frac \gamma 2},
\end{align*}
which shows the desired result.
\end{proof}

\subsection{Strong convergence rate of full discretization}
In this part, we extend the interpolation approach to the study of strong convergence rate of full discretization.  
For convenience, we denote $Z^N_k=Z^N(t_{k})$, $k\le K$.
We recall the equivalent form of \eqref{full0}, i.e., $X^N_k=Y^N_k+Z^N_k, k\le K$, where 
\begin{align*}
Y^N_{k+1}&=Y^N_{k}-A^2Y^N_{k+1}\delta t
-AP^Nf(Y^N_{k+1}+Z^N_{k+1})\delta t,\\ 
dZ^N&=-A^2Z^Ndt
+P^NdW(t)
\end{align*}
with $k\le K-1$, $Y^N_0=X^N(0)$ and $Z^N_0=0$. 
To make sure that the implicit method is solvable, we take $\delta t<\min(1,\frac 1{({L_{f}}-\lambda_1)\lor 0})$. 
Since for $t_k=k\delta t$, 
\begin{align*}
\|X^N_k-X(t_k)\|&\le \|X^N_k-X^N(t_k)\|+\|X^N(t_k)-X(t_k)\|\\
&\le  \|Y^N_k-Y^N(t_k)\|+\|X^N(t_k)-X(t_k)\|,
\end{align*}
it suffices to estimate the first term $\|Y^N_k-Y^N(t_k)\|$.
To apply the  interpolation approach, we need the optimal regularity estimate of $Y^N_k$ and the sharp strong convergence analysis of $\|Y^N_k-Y^N(t_k)\|_{\HH^{-1}}.$

\begin{lm}\label{pri-ynk}
Let $(I-\mathbb L) X_0\in \HH^{\gamma}$, $\gamma \in (0,\frac 32)$, $p\ge1$.
Then $Y^N_k$  satisfies 
\begin{align*}
\E\Big[\sup_{k\le K}\big\|(I-\mathbb L)Y^N_k\big\|_{\HH^{\gamma}}^p\Big]
&\le C(X_0,T,p),
\end{align*}
and 
\begin{align*}
\E\Big[\big\|Y^N_k-Y^N_{k_1}\big\|^p\Big]
&\le C(X_0,T,p)|(k-k_1)\delta t|^{\frac {\gamma p} 4}
\end{align*}
for a positive constant $ C(X_0,T,p)$ and $k_1,k\le K$.
\end{lm}
\begin{proof}
The proof is similar to that of Lemma \ref{pri-yn}. 
Since 
\begin{align}\label{ily}
(I-\mathbb L)Y^N_{k+1}=&(I-\mathbb L)Y^N_{k}-A^2\delta t(I-\mathbb L)Y^N_{k+1}\nonumber\\
&-A\delta t(I-\mathbb L)P^N(f(Y^N_{k+1}+Z^N_{k+1})), 
\end{align}
taking inner product with $(I-\mathbb L)Y^N_{k+1}$ in $\HH^{-1}$ on both sides 
leads to
\begin{align*}
\|(I-\mathbb L)Y^N_{k+1}\|_{\HH^{-1}}^2
&\le \|(I-\mathbb L)Y^N_{k}\|_{\HH^{-1}}^2
-2\|\nabla (I-\mathbb L)Y^N_{k+1}\|^2\delta t\\
&\quad-2\<(I-\mathbb L)f(Y^N_{k+1}+Z^N_{k+1}),(I-\mathbb L)Y^N_{k+1}\>\delta t.
\end{align*}
From the monotonicity of $-f$, the equivalence of norms in $\HH^1$ and $H^1$ for functions in $\HH^1$, and the Young inequality, it follows that for some small $\epsilon>0$,
\begin{align*}
&\|(I-\mathbb L)Y^N_{k+1}\|_{\HH^{-1}}^2
+8(c_4-\epsilon)\|(I-\mathbb L)Y^N_{k+1}\|_{L^4}^4\delta t+
(2-\epsilon)\|(I-\mathbb L)Y^N_{k+1}\|_{\HH^{1}}^2\delta t\\
&\le 
\|(I-\mathbb L)Y^N_{k}\|_{\HH^{-1}}^2
+C(\epsilon)\Big(1+\|\mathbb L Y_{k+1}^N\|_{L^{4}}^{4}
+\|Z_{k+1}^N\|_{L^{4}}^{4}\Big)\delta t.
\end{align*}
By taking the $p$th moment, the priori estimate of $Z_k^N$ and the fact that $\mathbb L Y^N_k=\mathbb L X^N(0)$, we have 
\begin{align*}
&\E\Big[\sup_{k\le K}\|(I-\mathbb L)Y^N_{k}\|_{\HH^{-1}}^{2p}\Big]
+\E\Big[\Big(\sum_{k=0}^{K-1}8(c_4-\epsilon)\|(I-\mathbb L)Y^N_{k+1}\|_{L^4}^4\delta t\Big)^p\Big]\\
&+\E\Big[\Big(\sum_{k=0}^{K-1}(2-\epsilon)\|(I-\mathbb L)Y^N_{k+1}\|_{\HH^{1}}^2\delta t\Big)^p\Big]\\
&\le C(\epsilon,p,T)\sum_{k=0}^{K-1}\E\Big[\Big(1+\|\mathbb L Y_{k+1}^N\|_{L^{4}}^{4}
+\|Z_{k+1}^N\|_{L^{4}}^{4}\Big)^p\Big]\delta t\\
&\le C(\epsilon,p,T,X_0).
\end{align*}
Next we show the boundedness of $(I-\mathbb L)Y_k^N$ in $\HH$.
By taking inner product on both sides of equation \eqref{ily} with $(I-\mathbb L)Y^N_{k+1}$ in $\HH$,
we obtain 
\begin{align*}
\|(I-\mathbb L)Y^N_{k+1}\|^2
&\le \|(I-\mathbb L)Y^N_{k}\|^2
-(2-\epsilon)\delta t\|A(I-\mathbb L)Y^N_{k+1}\|^2\\
&\quad -(24c_{4}-\epsilon)\|Y^N_{k+1}\|^{2}\|\nabla(I-\mathbb L)Y^N_{k+1}\|^2\delta t
\\
&\quad +C(\epsilon)
(\|\nabla(I-\mathbb L)Y^N_{k+1}\|^2+\|Y^N_{k+1}\|_{L^{4}}^{4}+\|Z^N_{k+1}\|_{L^4}^4)\\
&\qquad+\|\nabla Z^N_{k+1}\|_{L^4}^{4}+\|\nabla Z^N_{k+1}\|^2)\delta t.
\end{align*}
Thus it is concluded that for $p\ge 1$,
\begin{align*}
&\E\Big[\sup_{k\le K}\|(I-\mathbb L)Y^N_{k+1}\|^{2p}\Big]+(2-\epsilon)\E\Big[\Big(\sum_{k=0}^{K-1}\|A(I-\mathbb L)Y^N_{k+1}\|^2\delta t\Big)^{p}\Big]\\
&\le C(\epsilon, T)\E\Big[\Big(\sum_{k=0}^{K-1}\|\nabla(I-\mathbb L)Y^N_{k+1}\|^2\delta t\Big)^{p}\Big]+C(\epsilon, T)\E\Big[\Big(\sum_{k=0}^{K-1}\|Y^N_{k+1}\|_{L^{4}}^{4}\delta t\Big)^{p}\Big]
\\
&\quad+C(\epsilon, T)\sum_{k=0}^{K-1}\E\Big[\|Z^N_{k+1}\|_{L^4}^{4p}+\| \nabla Z^N_{k+1}\|_{L^4}^{4p}+\|\nabla Z^N_{k+1}\|^{2p}\Big]\delta t\le C(X_0,p,T).
\end{align*}

Similar to the proof of Proposition \ref{prop-spa}, we need the boundedness of $\|(I-\mathbb L)Y_k^N\|_{L^6}$.
From the Sobolev embedding theorem, the smoothing effect of $T_{\delta t}$, the Gagliardo--Nirenberg and Young inequalities, it follows  that 
\begin{align*}
&\|(I-\mathbb L)Y^N_{k+1}\|_{L^6}\\
&\le \left\|T_{\delta t}^{k+1}(I-\mathbb L)Y_0^N\right\|_{L^6}
+ \left\|\sum_{j=0}^{k}T_{\delta t}^{k+1-j}A(I-\mathbb L)f(Y_{j+1}^N+Z_{j+1}^N)\right\|_{L^6}\delta t\\
&\le C((k+1)\delta t)^{\min(-\frac 1{12}+\frac \gamma 4,0)}\|Y_0^N\|_{\mathbb H^\gamma}+
C\sum_{j=0}^k(k+1-j)^{-\frac 7{12}}\delta t^{-\frac 7{12}}
\Big(\|Z_{j+1}^N\|_{L^{6}}^{3}\\
&\qquad+\|\mathbb L Y^N_{j+1}\|_{L^{6}}^{3}
+\|(I-\mathbb L)Y^N_{j+1}\|\Big)\delta t\\
&\quad+C\sum_{j=0}^k(k+1-j)^{-\frac 79}\delta t^{-\frac 79} \sup_{j\le K}\|(I-\mathbb L)Y^{N}_j\|^{\frac {10}3}\delta t+\sum_{j=0}^k\|(-A)(I-\mathbb L)Y^N_{j+1}\|^2\delta t.
\end{align*}
The a priori estimates of $Y^N_k$
 and $Z^N_k$ yield that for $p\ge 1$
\begin{align*}
\|(I-\mathbb L)Y^N_{k}\|_{L^p(\Omega;L^6)}
&\le C(T,X_0,p)(1+((k+1)\delta t)^{\min(-\frac 1{12}+\frac \gamma 4,0)}).
\end{align*}
Now, we are in the position to give the desired regularity estimate. 
From the mild form of $Y^N_k$, the Minkowski inequality and the above a priori estimates, we have 
\begin{align*}
&\big\|\sup_{k\le K}\|(I-\mathbb L)Y^N_k\|_{\HH^{\gamma}}\big\|_{L^p(\Omega)}\\
&\le 
C\|(I-\mathbb L)Y^N_0\|_{L^p(\Omega;\HH^{\gamma})}
+C\delta t \sup_{k\le K}\sum_{j=0}^{k-1} \left\|T_{\delta t}^{k-j}A(I-\mathbb L)f(Y^N_{j+1}+Z^N_{j+1})\right\|_{L^p(\Omega;\HH^{\gamma})}\\
&\le C(p)\|(I-\mathbb L)X^N(0)\|_{\HH^{\gamma}}^p\\
&\quad+
C(p)\sup_{k\le K}\sum_{j=0}^{k-1}(t_k-t_j)^{-\frac 12-\frac \gamma 4}\Big(1+\|Y^N_{j+1}\|_{L^{3p}(\Omega;L^6)}^3+\|Z^N_{j+1}\|_{L^{3p}(\Omega;L^6)}^3\Big)\delta t\\
&\le C(p)\|(I-\mathbb L)X^N(0)\|_{\HH^{\gamma}}^p+
C(T,p)\sup_{k\le K}\sum_{j=0}^{k-1}(t_k-t_j)^{-\frac 12-\frac \gamma 4}t_{j+1}^{\min(-\frac 14+\frac {3}4\gamma,0)}\delta t\\
&\le C(T,p,X_0).
\end{align*}
For convenience, we assume that  $k> k_1$.
Similar arguments in the proof of Proposition 
\ref{prop-tm}  yield that 
\begin{align*}
\E\Big[\big\|Y^N_k-Y^N_{k_1}\big\|^p\Big]
\le& \E\Big[\|(I-\mathbb L)T_{\delta t}^{k_1}(T_{\delta t}^{k-k_1}-I)Y_0^N\|^p\Big]\\
&
+\sum_{j=0}^{k_1-1}\E\Big[\Big\|(T_{\delta t}^{k_1-j}(T_{\delta t}^{k-k_1}-I)A(I-\mathbb L)f(Y^N_{j+1}+Z^N_{j+1})\Big\|^p\Big]\delta t
\\
&+\sum_{j=k_1}^{k-1}\Big\|T_{\delta t}^{k-j}Af(Y^N_{j+1}+Z^N_{j+1})\Big\|\delta t\\
\le& C(X_0,T,p)|(k-k_1)\delta t|^{\frac {\gamma p} 4}.
\end{align*}
Combining the above regularity estimates together, we finish the proof.
\end{proof}

\begin{cor}\label{rk-est-ynk}
Let $(I-\mathbb L) X_0\in \HH^{\gamma}$, $\gamma \in (0,\frac 32)$, $p\ge 1$.  Then $Y^N_k$,$1 \le k\le K$ satisfies that for any small $\epsilon>0$,
\begin{align}\label{rk-est-ynk-e}
\big\| (I-\mathbb L)Y^N_k\big\|_{L^p(\Omega;E)}
&\le C(X_0,T,p)(1+t_{k}^{\min(-\frac 18-\epsilon+\frac \gamma 4,0)}).
\end{align}
\end{cor}

To deduce the strong convergence rate in time, 
we  introduce an auxiliary process $\widetilde Y^N_k$, $k\le K$ with $Y^N_0=X^N(0)$, defined by 
\begin{align*}
\widetilde Y^N_{k+1}
=\widetilde Y^N_{k}
-A^2\delta t \widetilde Y^N_{k+1}
-P^N A f(Y^N(t_{k+1})+Z^N_{k+1})\delta t.
\end{align*}
Then we split the error of $Y^N_{k}-Y^N(t_{k})$ as 
\begin{align*}
\|Y^N_{k}-Y^N(t_{k})\|\le \|Y^N(t_{k})-\widetilde Y^N_{k}\|
+\|\widetilde Y^N_{k}-Y^N_{k}\|.
\end{align*} 
The first error is bounded as the following lemma.
The second error will be dealt with the interpolation arguments.

\begin{lm}\label{lm-ynk}
Let $(I-\mathbb L) X_0\in \HH^{\gamma}$, $\gamma \in (0,\frac 32)$, $N\in \N^+$ and  $p\ge1$. For $k\in \mathbb N^+, k\le K, T=K\delta t,$
there exists a positive constant $ C(X_0,t_k,p)$ such that \begin{align}\label{strong-ynk}
\big\|Y^N(t_k)-\widetilde Y^N_{k}\big\|_{L^p(\Omega;\HH)}
&\le C(X_0,t_k,p)\delta t^{\frac \gamma 2}.
\end{align}
\end{lm}

\begin{proof}

Denote $[s]_{\delta t}:=\max\{0,\delta t,\cdots,k\delta t,\cdots \}\cap [0,s]$ and $[s]=\frac {[s]_{\delta t}} {\delta t}$.
The mild forms of $Y^N(t_{k})$ and $\widetilde Y^N_{k}$ yield that 
\begin{align}\label{dec}
&\|Y^N(t_{k})-\widetilde Y^N_{k}\|_{L^p(\Omega;\HH)}\\\nonumber 
&\le
\Big\|\int_0^{t_{k}}(e^{-A^2(t_k-s)}-T_{\delta t}^{k-[s]})AP^Nf(Y^N(s)+Z^N(s))ds\Big\|_{L^p(\Omega;\HH)}\\\nonumber
&\quad+\Big\|\int_0^{t_{k}}T_{\delta t}^{k-[s]}AP^N\Big(f(Y^N(s)+Z^N(s))-f(Y^N({[s]_{\delta t}+\delta t})+Z^N_{[s]+1})\Big)ds\Big\|_{L^p(\Omega;\HH)}\\\nonumber
&:=I_1+I_2.
\end{align}
By the properties of $e^{-A^2t}$ and $T_{\delta t}^{k}$, the priori estimates of $Y^N$ and $Z^N$ and Lemma \ref{nabla}, the first term is estimated as follows.
For $\gamma\in (0,1)$, $\beta\in (0,\min(\frac 14+\frac {3\gamma}4,\frac 12))$,
\begin{align*}
I_1
&\le  C\sum_{j=0}^{k-1}\int_{t_j}^{t_{j+1}}(t_k-[s]_{\delta t})^{-\frac 12-\beta}\delta t^{\beta}
(1+\|X^N(s)
\|^3_{L^{3p}(\Omega;L^6)})ds\\
&\le C\sum_{j=0}^{k-1}\int_{t_j}^{t_{j+1}}(t_k-[s]_{\delta t})^{-\frac 12-\beta}\delta t^{\beta}
(1+s^{\min(-\frac 14+\frac {3\gamma }4,0)})ds\\
&\le C(T,X_0,p)\delta t^{\beta}.
\end{align*}
For  $1\le \gamma<\frac 32$ and $\beta\in (0,\frac 12)$, we have that
\begin{align*}
I_1
&\le  C\sum_{j=0}^{k-1}\int_{t_j}^{t_{j+1}}(t_k-[s]_{\delta t})^{-\frac 12+\beta}\delta t^{\beta+\frac \gamma4}
\Big\|\Big(1+\|X^N(s)\|_E^2+\|\nabla X^N(s)\|_E^2
\Big)\\
&\quad \|X^N(s)\|_{H^{\gamma}}\Big\|_{L^p(\Omega;\R)}ds\\
&\le C\sum_{j=0}^{k-1}\int_{t_j}^{t_{j+1}}(t_k-[s]_{\delta t})^{-\frac 12+\beta}\delta t^{\beta+\frac \gamma4}
\Big(1+s^{-\frac 34-\epsilon+\frac \gamma 2}\Big)ds\\
&\le C(T,X_0,p)\delta t^{\beta+\frac \gamma4}.
\end{align*}

Similarly, by using the Talyor formula and the mild form of $X^N$, we have 
\begin{align*}
I_{2}
&\le \Big\|\int_0^{t_{k}}T_{\delta t}^{k-[s]}AP^N\Big(f'(X^N(s))(e^{-A^2([s]_{\delta t}+\delta t-s)}-I)X^N(s)\Big)ds\Big\|_{L^p(\Omega;\HH)}\\
& +\Big\|\int_0^{t_{k}}T_{\delta t}^{k-[s]}AP^N\Big(f'(X^N(s))\int_{s}^{[s]_{\delta t}+\delta t}e^{-A^2([s]_{\delta t}+\delta t-r)}\\
&\qquad AP^Nf(X^N(r))dr\Big)ds\Big\|_{L^p(\Omega;\HH)}\\
& +\Big\|\int_0^{t_{k}}T_{\delta t}^{k-[s]}AP^N\Big(f'(X^N(s))\int_{s}^{[s]_{\delta t}+\delta t}e^{-A^2([s]_{\delta t}+\delta t-r)}P^NdW(r)\Big)ds\Big\|_{L^p(\Omega;\HH)}\\
& + \Big\|\int_0^{t_{k}}
T_{\delta t}^{k-[s]}AP^N\Big(\int_0^{1}f''(\lambda X^N(s)+(1-\lambda)X^N([s]_{\delta t}+\delta t))d\lambda\\
&\qquad \big(X^N(s)-X^N([s]_{\delta t}+\delta t)\big)^2\Big)ds\Big\|_{L^p(\Omega;\HH)}
:=I_{21}+I_{22}+I_{23}+I_{24}.
\end{align*}
Then the smoothing effect of $e^{-A^2t}$, Lemma \ref{nabla} and Corollary  \ref{rk-sob} yield that 
for $\gamma\in (0,1]$, $\beta\in (0,1)$, 
\begin{align*}
I_{21}
&\le  \Big\|\int_0^{t_{k}}T_{\delta t}^{k-[s]}A^{1+\frac 12} A^{-\frac 12}\Big((I-\mathbb L)f'(X^N(s))(e^{-A^2([s]_{\delta t}+\delta t-s)}-I)X^N(s)\Big)ds\Big\|_{L^p(\Omega;\HH)}\\
&\le \int_0^{t_{k}}(t_k-[s]_{\delta t})^{-\frac 34}\|(I-\mathbb L)f'(X^N(s))\\
&\qquad (e^{-A^2([s]_{\delta t}+\delta t-s)}-I)X^N(s)\|_{L^p(\Omega;\HH^{-1})}ds\\
&\le   \int_0^{t_{k}}(t_k-[s]_{\delta t})^{-\frac 34}\|(1+\|X^N(s)\|_E^2\\
&\qquad+\Big\|(I-\mathbb L)X^N(s)\|_{\HH^{\beta}}^2)
\|e^{-A^2([s]_{\delta t}+\delta t-s)}-I)X^N(s)\|_{\HH^{-\beta}}\Big\|_{L^p(\Omega;\R)}ds\\
&\le C(T,X_0,p)\delta t^{\frac \beta 4+\frac \gamma 4}.
\end{align*}
and for $\gamma\in (1,\frac 32)$, and any sufficiently small $\epsilon>0,$
\begin{align*}
I_{21}
&\le  \Big\|\int_0^{t_{k}}T_{\delta t}^{k-[s]}A^{1+\frac 34+\epsilon} A^{-\frac 34-\epsilon}\Big((I-\mathbb L)f'(X^N(s))\\
&\qquad (e^{-A^2([s]_{\delta t}+\delta t-s)}-I)X^N(s)\Big)ds\Big\|_{L^p(\Omega;\HH)}\\
&\le \int_0^{t_{k}}(t_k-[s]_{\delta t})^{-\frac 78-\epsilon}\|(I-\mathbb L)f'(X^N(s))\\
&\qquad (e^{-A^2([s]_{\delta t}+\delta t-s)}-I)X^N(s)\|_{L^p(\Omega;\HH^{-\frac 32-2\epsilon })}ds\\
&\le   \int_0^{t_{k}}(t_k-[s]_{\delta t})^{-\frac 78-\epsilon}\|(1+\|X^N(s)\|_E^2+\|\nabla X^N(s)\|_E^2\\
&\qquad+\Big\|(I-\mathbb L)X^N(s)\|_{\HH^{\gamma}}^2)
\|e^{-A^2([s]_{\delta t}+\delta t-s)}-I)X^N(s)\|_{\HH^{-\gamma}}\Big\|_{L^p(\Omega;\R)}ds\\
&\le C(t_k,X_0,p)\delta t^{\frac \gamma 2},
\end{align*}
where $C(t_k,X_0,p)\to \infty$ if $t_k\to 0$.
In particular, if $\gamma>\frac 54$, we have the uniform control $C(t_k,X_0,p)\le C(T,X_0,p).$ 

Similarly, we have that  for $\gamma\in (\frac 15,\frac 32)$ and any $\beta \in (0,1)$, 
\begin{align*}
I_{22}%&= \Big\|\int_0^{t_{k}}T_{\delta t}^{k-[s]}AP^N\Big(f'(X^N(s))\int_{s}^{[s]_{\delta t}+\delta t}e^{-A^2([s]_{\delta t}+\delta t-r)}\\\
%&\qquad AP^Nf(X^N(r))dr\Big)ds\Big\|_{L^p(\Omega;\HH)}\\
&\le C\int_0^{t_{k}} (t_k-[s]_{\delta t})^{-\frac 34}\Big\|(1+\|X^N(s)\|_E^2+\|X^N(s)\|_{\HH^{\beta}}^2)\\
&\qquad \int_{s}^{[s]_{\delta t}+\delta t}\|e^{-A^2([s]_{\delta t}+\delta t-r)}A^{1-\frac \beta2}\| \|P^Nf(X^N(r))\|dr\Big\|_{L^p(\Omega;\R)}ds
\\
&\leq C\int_0^{t_{k}} (t_k-[s]_{\delta t})^{-\frac 34}(1+s^{\min(-\frac 14+\frac\gamma 2-\epsilon,0)})\delta t^{\frac 12+\frac \beta 4}s^{\min{(-\frac 14+\frac 34\gamma,0)}}ds\\
&\le C(T,X_0,p)\delta t^{\frac 12+\frac \beta 4},
\end{align*}
and for $\gamma\in (0,\frac 15]$, $\beta\in (0,\frac 12),$
\begin{align*}
I_{22}&\le C\int_0^{t_{k}} (t_k-[s]_{\delta t})^{-\frac 12}(1+s^{\min(-\frac 14+\frac\gamma 2-\epsilon,0)})\delta t^{\frac 12}s^{\min{(-\frac 14+\frac 34\gamma,0)}}ds\\
&\le C(T,X_0,p)\delta t^{\frac 12}.
\end{align*}
\iffalse
From the mild form of $Y^N$, $Z^N$ and $Z^N_k$, 
we have  the continuity in the negative Sobolev space as for $s\le t\le T$, $1\le \gamma<\frac 32$, 
\begin{align*}
&\|Y^N(t)-Y^N(s)\|_{ \HH^{-\gamma}}\\
&\le \|e^{-A^2s}(e^{-A^2(t-s)}-I)A^{-\frac \gamma 2}(I-\mathbb L)Y^N(0)\|\\
&\quad+\int_0^s\|e^{-A^2(s-r)}(e^{-A^2(t-s)}-I)A^{1-\frac \gamma 2}(I-\mathbb L)f(Y^N(r)+Z^N(r))\|dr\\
&\quad+ \int_s^t\|e^{-A^2(t-r)}A^{1-\frac \gamma 2}(I-\mathbb L)f(Y^N(r)+Z^N(r))\|dr\\
&\le C(t-s)^{\frac \gamma 2}\|(I-\mathbb L)Y^N(0)\|_{\HH^{\gamma}}\\
&\quad+C(t-s)^{\frac \gamma 2}\int_0^s (s-r)^{-\frac 12}\|(I-\mathbb L)f(Y^N(r)+Z^N(r))\|_{\HH^{\gamma}}dr\\
&\quad+\int_s^t\|e^{-A^2(t-r)}A^{1-\gamma}\|
\|(I-\mathbb L)f(Y^N(r)+Z^N(r))\|_{\HH^{\gamma}}dr.
\end{align*}
Taking the $L^p(\Omega;\R)$-norm leads to 
\begin{align*}
\|Y^N(t)-Y^N(s)\|_{ L^p(\Omega; \HH^{-\gamma})}
&\le C(T,X_0,p)\delta t^{\frac \gamma 2}.
\end{align*}
\fi
From the stochastic Fubini theorem, the Burkholder--Davis--Gundy inequality and H\"older's inequality, it follows that for $p\ge 2$, sufficiently small $\epsilon>0$ 
\begin{align*}
I_{23}&=\Big\|\sum_{j=0}^{k-1}\int_{t_j}^{t_{j+1}}T_{\delta t}^{k-[s]}Af'(X^N(s))\int_{s}^{t_{j+1}}e^{-A^2([s]_{\delta t}+\delta t-r)}P^NdW(r)ds\Big\|_{L^p(\Omega;\HH)}\\
&=\Big\|\sum_{j=0}^{k-1}\int_{t_j}^{t_{j+1}}\int_{t_j}^{r}T_{\delta t}^{k-[s]}Af'(X^N(s))e^{-A^2([s]_{\delta t}+\delta t-r)}P^NdsdW(r)\Big\|_{L^p(\Omega;\HH)}\\
&\le C_p\sqrt{\sum_{j=0}^{k-1}\int_{t_j}^{t_{j+1}}
\Big\|\int_{t_j}^{r}T_{\delta t}^{k-[s]}Af'(X^N(s))e^{-A^2(t_{j+1}-r)}P^Nds\Big\|_{L^p(\Omega;\LL_2^0)}^2dr}\\
&\le C_p\delta t^{\frac 12}\sqrt{\sum_{j=0}^{k-1}\int_{t_j}^{t_{j+1}}
\int_{t_j}^{r}\Big\|T_{\delta t}^{k-[s]}A^{1-\epsilon}A^{\epsilon}f'(X^N(s))e^{-A^2(t_{j+1}-r)}P^N\Big\|_{L^p(\Omega;\LL_2^0)}^2dsdr}.
\end{align*}
Then we have
\begin{align*}
I_{23} &\le C_p\delta t^{\frac 12}\Big(\sum_{j=0}^{k-1}\int_{t_j}^{t_{j+1}} (t_k-t_j)^{-1+\epsilon}\int_{t_j}^{r}\Big\|\Big(\|X^N(s)\|_{H^{2\epsilon}}^4+\|X^N(s)\|_{E}^4\Big)\\
&\quad \times \Big(\sum_{l=1}^N\Big\|e^{-A^2(t_{j+1}-r)}Q^{\frac 12}e_l\Big\|_{H^{2\epsilon}}^2+ \sum_{l=1}^N\Big\|e^{-A^2(t_{j+1}-r)}Q^{\frac 12}e_l\Big\|_E^2\Big)\Big\|_{L^{\frac p2}(\Omega)}dsdr\Big)^{\frac 12}\\
&\le C_p\delta t^{\frac 12}\Big(\sum_{j=0}^{k-1}\int_{t_j}^{t_{j+1}} (t_k-t_j)^{-1+\epsilon}\int_{t_j}^{r}(1+s^{\min{(-\frac 12+\gamma,0)}})(t_{j+1}-r)^{-\frac 14-\epsilon}dsdr\Big)^\frac 12\\
&=C_p\delta t^{\frac 12}\Big(\sum_{j=0}^{k-1} (t_k-t_j)^{-1+\epsilon}
\int_{t_j}^{t_{j+1}}(t_{j+1}-r)^{-\frac 14-\epsilon}
\int_{t_j}^{r}(1+s^{\min{(-\frac 12+\gamma,0)}})dsdr\Big)^\frac 12.
\end{align*}
For $\gamma\in(\frac 12,\frac 32),$
\begin{align*}
I_{23}&\le C(T,X_0,p)\delta t^\frac 12\Big(\sum_{j=0}^{k-1} (t_k-t_j)^{-1+\epsilon}\delta t\int_{t_j}^{t_{j+1}}(t_{j+1}-r)^{-\frac 14-\epsilon}dr\Big)^\frac 12\\
&\leq C(T,X_0,p)\delta t^\frac 12\delta t^{\frac 38-\frac \epsilon 2}.
\end{align*}
For $\gamma\in(0,\frac 12],$
\begin{align*}
I_{23}&=\Big\|\sum_{j=0}^{k-1}\int_{t_j}^{t_{j+1}}T_{\delta t}^{k-[s]}Af'(X^N(s))\int_{s}^{t_{j+1}}e^{-A^2([s]_{\delta t}+\delta t-r)}P^NdW(r)ds\Big\|_{L^p(\Omega;\HH)}\\
&\le C \sum_{j=0}^{k-1}\int_{t_j}^{t_{j+1}}(t_k-[s]_{\delta t})^{-\frac 12}\Big\|(\|X^N(s)\|_E^2+1)\int_{s}^{t_{j+1}}e^{-A^2([s]_{\delta t}+\delta t-r)}P^NdW(r)\Big\|_{L^p(\Omega;\HH)}ds\\
&\le C\sum_{j=0}^{k-1}\int_{t_j}^{t_{j+1}}(t_k-[s]_{\delta t})^{-\frac 12}s^{-\frac 14+\frac \gamma 2-\epsilon}(\int_s^{t_j+1}([s]_{\delta t}+\delta t-r)^{-\frac 14}dr)^\frac 12 ds\le C \delta t^{\frac 38}.
\end{align*}
Due to the continuity of $Z^N$ and $Y^N$ and the Sobolev embedding theorem, we obtain for small $\epsilon>0$,
\begin{align*}
I_{24}&=\Big\|\int_0^{t_{k}}
T_{\delta t}^{k-[s]}A^{1+\frac 14+\epsilon}A^{-\frac 14-\epsilon}P^N\Big(\int_0^{1}f''(\lambda X^N(s)+(1-\lambda)X^N([s]_{\delta t})(1-\lambda)d\lambda\\
&\qquad \big(X^N(s)-X^N([s]_{\delta t}+\delta t)\big)^2\Big)ds\Big\|_{L^p(\Omega;\HH)}\\
&\le
C\int_0^{t_{k}}(t_k-[s]_{\delta t})^{-\frac 58+\frac \epsilon 2}\Big\|\int_0^{1}f''(\lambda X^N(s)+(1-\lambda)X^N([s]_{\delta t}+\delta t))d\lambda\Big\|_{L^{2p}(\Omega;E)}\\
&\qquad \Big\|X^N(s)-X^N([s]_{\delta t}+\delta t)\Big\|^2_{L^{4p}(\Omega;\HH)}ds\le C(T,X_0,p)\delta t^{\frac \gamma 2}.
\end{align*}
Combining \eqref{dec} and the above regularity estimates, we complete the proof.
\end{proof}

From the above proof, it is not difficult to see that if the initial data $X_0$ is smooth enough, such as $\gamma>\frac 54,$ then the constant in $C(X_0,t_k,p)$ in Lemma \eqref{lm-ynk} is bounded by $C(X_0,T,p)$ independent of $k.$

Next, we show the optimal regularity of  $\widetilde Y^N_{k}$ and 
the optimal convergence analysis of $\widetilde Y^N_{k}-Y^N_k$ in $\HH^{-1}$.
\begin{lm}\label{lm-pri-y}
Let $(I-\mathbb L) X_0\in \HH^{\gamma}$, $\gamma \in (0,\frac 32)$, $p\ge1$.
Then $\widetilde Y^N_k$  satisfies 
\begin{align*}
\E\Big[\sup_{k\le K}\big\|(I-\mathbb L)\widetilde Y^N_k\big\|_{\HH^{\gamma}}^p\Big]
&\le C(X_0,T,p).
\end{align*}
\end{lm}

\begin{lm}\label{lm-ynk1}
Assume that $(I-\mathbb L) X_0\in \HH^{\gamma}$, for all $\gamma \in (0,\frac 32)$, $N\in \N^+$ and  $p\ge1$. For $k\in \mathbb N^+, k\le K, T=K\delta t,$ there 
exist $\delta t_0\le 1$ and  $ C(X_0,T,p)>0$ such that 
for any $\delta t\le \delta t_0$, $N\in \N^+$, we have 
\begin{align}\label{strong-ynk-h}
\big\|\widetilde Y^N_{k}-Y^N_k\big\|_{L^p(\Omega;\HH^{-1})}
&\le C(X_0,T,p)\delta t^{\frac \gamma 2}.
\end{align} 
\end{lm}
\begin{proof}
From the definitions of  $Y^N_{k+1}$ and $\widetilde Y^N_{k+1}$, it follows that for some $\epsilon<2,$
\begin{align*}
&\quad\|Y^N_{k+1}-\widetilde Y^N_{k+1}\|_{\HH^{-1}}^2\\
&\le 
\|Y^N_{k}-\widetilde Y^N_{k}\|_{\HH^{-1}}^2
-2\|A^{\frac 12}(Y^N_{k+1}-\widetilde Y^N_{k+1})\|^2 \delta t \\
&\quad
-2\<P^N(f(X^N_{k+1})-f(Y^N(t_{k+1})+Z^N_{k+1})),(Y^N_{k+1}-\widetilde Y^N_{k+1})\>\delta t\\
&\le \|Y^N_{k}-\widetilde Y^N_{k}\|_{\HH^{-1}}^2
-2\|A^{\frac 12}(Y^N_{k+1}-\widetilde Y^N_{k+1})\|^2 \delta t \\
&\quad
-2\<f(X^N_{k+1})
-f(\widetilde Y^N_{k+1}+Z^N_{k+1}),(Y^N_{k+1}-\widetilde Y^N_{k+1})\>\delta t
\\
&\quad -
2\<f(\widetilde Y^N_{k+1}+Z^N_{k+1})-f(Y^N(t_{k+1})+Z^N_{k+1}),(Y^N_{k+1}-\widetilde Y^N_{k+1})\>\delta t\\
&\le \|Y^N_{k}-\widetilde Y^N_{k}\|_{\HH^{-1}}^2
-(2-\epsilon)\|A^{\frac 12}(Y^N_{k+1}-\widetilde Y^N_{k+1})\|^2 \delta t\\
&\quad+2C(\epsilon)\|(I-\mathbb L)(f(\widetilde Y^N_{k+1}+Z^N_{k+1})-f(Y^N(t_{k+1})+Z^N_{k+1}))\|^2_{\HH^{-1}}\delta t.
\end{align*}
By using the a priori estimates of $Y_{k+1}^N$, $X^N$ and $Z_{k+1}^N$, Lemmas \ref{sob} and \ref{lm-ynk}, we have that for any $\beta\in (0,1)$,
\begin{align*}
&\|Y^N_{k+1}-\widetilde Y^N_{k+1}\|_{L^{2p}(\Omega;\HH^{-1})}^2\\
&\le C\sum_{j=0}^{k-1}\|(I-\mathbb L)(f(\widetilde Y^N_{j+1}+Z^N_{j+1})-f(Y^N(t_{j+1})+Z^N_{j+1}))\|^2_{L^{2p}(\Omega;\HH^{-1})}\delta t\\
&\le C(X_0,p)\sum_{j=0}^{k-1}\Bigg\|(1+\|Y^N(t_{j+1})\|_{E}^4+\|\widetilde Y^N_{j+1}\|_{E}^4+\|Z^N_{j+1}\|_{E}^4+\|Y^N(t_{j+1})\|_{\HH^{\beta}}^4\\
&\quad+\|\widetilde Y^N_{j+1}\|_{\HH^{\beta}}^4+\|Z^N_{j+1}\|_{\HH^{\beta}}^4)\Bigg\|_{L^{4p}(\Omega; \mathbb R)}\|\widetilde Y^N_{j+1}-Y^N(t_{j+1})\|_{L^{4p}(\Omega; \HH^{-\beta})}^2\delta t\\
&\le C(X_0,T,p)\delta t^{\gamma},
\end{align*}
which completes the proof.
\end{proof}
Based on the interpolation method,  Lemmas \ref{lm-pri-y} and  \ref{lm-ynk1}, we obtain the following convergence result.

\begin{prop}\label{prop-str}
Under the condition of Lemma \ref{lm-ynk1},  there 
exist $\delta t_0\le 1$ and  $ C(X_0,T,p)>0$ such that 
for any $\delta t\le \delta t_0$, $N\in \N^+$,  we have 
\begin{align*}
\big\|Y^N(t_k)-Y^N_k\big\|_{L^p(\Omega;H)}
&\le C(X_0,T,p)\delta t^{\frac \gamma 2 -\epsilon},
\end{align*} 
where $\epsilon>0$ is sufficient small.
\end{prop}
\begin{proof}
By the mild form of $\widetilde Y^N_k$ and $Y^N_k$ and Lemma \ref{sob},
we have for any $\beta<1$,
\begin{align*}
&\|\widetilde Y^N_k-Y^N_k\|\\
&\le 
\sum_{j=0}^{k-1}\|\int_{t_j}^{t_{j+1}}T_{\delta t}^{k-[s]}P^NA^{\frac 32}A^{-\frac 12}(f(Y^N([s]_{\delta t})+Z_{[s]}^N)-f(Y^N_{[s]}+Z^N_{[s]}))ds\|\\
&\le C\sum_{j=0}^{k-1}(t_k-t_{j})^{-\frac 34}
\big(1+\|Y^N(t_j)+Z_{j}^N\|_E^2+\|Y^N_{j}+Z^N_{j}\|_E^2+\|Y^N(t_j)+Z_{j}^N\|_{\HH^{\beta}}^2\\
&\quad +\|Y^N_{j}+Z^N_{j}\|_{\HH^{\beta}}^2\big)\|Y^N_j-Y^N(t_j)\|_{\HH^{-\beta}}
\delta t .
\end{align*}
Together with  Lemmas \ref{lm-ynk}, \ref{lm-pri-y} and \ref{lm-ynk1},  we obtain  for any $\beta\in (0,1)$ and  any sufficient small $\epsilon$,
\begin{align*}
\big\|Y^N(t_k)-Y^N_k\big\|_{L^p(\Omega;\HH^{-\beta})}
&\le \big\|Y^N(t_k)-\widetilde Y^N_k\big\|_{L^p(\Omega;\HH^{-\beta})}
+\big\|\widetilde Y^N_k-Y^N_k\big\|_{L^p(\Omega;\HH^{-\beta})}\\
&\le C(T,X_0,p)\delta t^{\frac \gamma 2 -\epsilon},
\end{align*}
which implies that 
\begin{align*}
\big\|Y^N(t_k)-Y^N_k\big\|_{L^p(\Omega;H)}
&\le \big\|Y^N(t_k)-\widetilde Y^N_k\big\|_{L^p(\Omega;H)}
+\big\|\widetilde Y^N_k-Y^N_k\big\|_{L^p(\Omega;H)}\\
&\le C(T,X_0,p)\delta t^{\frac \gamma 2-\epsilon}.
\end{align*}

\end{proof}

\begin{tm}\label{tm-strong1}
Let $(I-\mathbb L) X_0\in \HH^{\gamma}$ for all $\gamma \in (0,\frac 32)$, $N\in \N^+$ and  $p\ge1$.
There 
exist $\delta t_0\le 1$ and  $ C(X_0,T,p)>0$ such that 
for any $\delta t\le \delta t_0$, $N\in \N^+$, we have
\begin{align}\label{strong1}
\big\|X^N_k-X(t_k)\big\|_{L^p(\Omega;{H})}
&\le C(X_0,T,p)(\delta t^{\frac \gamma 2-\epsilon}+\lambda_N^{-\frac \gamma 2}),
\end{align}
where $\epsilon>0$ is sufficient small.
\end{tm}
\begin{proof}
Notice that 
\begin{align*}
\|X^N_k-X(t_k)\|&\le \|X^N_k-X^N(t_k)\|+\|X^N(t_k)-X(t_k)\|\\
&\le  \|Y^N_k-Y^N(t_k)\|+\|X^N(t_k)-X(t_k)\|.
\end{align*}
By using Lemma \ref{lm-ynk}, Proposition \ref{prop-str} and Theorem \ref{tm-strong0}, we complete the proof. 
\end{proof}

We also remark that 
the interpolation approach is also available for 
the numerical analysis on the strong convergence rates of the finite element method and the implicit Euler method for equation \eqref{spde} and will be studied further.

\section{Application to the cases of general noise and high dimension}
\label{sec-oth}
Now we extend this approach to study the  cases of  general 
noise and high dimension. For convenience, we assume that $\gamma\in (0,4]$.
In the case that 
$d=1$, $\gamma>0$, $\|A^{\frac {\gamma-2}2}Q^{\frac 12}\|_{\LL_2}<\infty$, 
we assume that $Q$ commutes with $A$ or $\gamma>\frac 12$.
The result of the case $\gamma>4$ is similar, we omit the details.

\begin{lm}\label{reg-tr}
Let $d=1$, $\gamma\in (0,4]$, $\|A^{\frac {\gamma-2}2}Q^{\frac 12}\|_{\LL_2}<\infty$,  $(I-\mathbb L) X_0\in \HH^{\gamma}$ and $p\ge1$.
Assume that $Q$ commutes with $A$ or $\gamma>\frac 12$, then
the unique mild solution  $X$  of equation \eqref{spde}  satisfies 
\begin{align}\label{reg-spa1}
\sup_{t\in [0,T]}\E\Big[\big\|(I-\mathbb L)X(t)\big\|_{\HH^{\gamma}}^p\Big]
&\le C(X_0,T,p)
\end{align}
and 
\begin{align}\label{reg-tm1}
\E\Big[\big\|X(t)-X(s)\big\|^p\Big]
&\le C(X_0,T,p)(t-s)^{\min(\frac {\gamma } 4,\frac 12)p}
\end{align}
for a positive constant $ C(X_0,T,p)$ and $0\le s\le t\le T$.
\end{lm}
\begin{proof}
The conditions on $Q$ and $A$ ensure that 
\begin{align*}
\E\big[\sup\limits_{t\in[0,T]}\|Z(t)\|_E^p\big]+\sup_{t\in[0,T]}\E\big[\|(I-\mathbb L)Z(t)\|_{\HH^{\gamma}}^p\big]\le C(T,p).
\end{align*}
According to \eqref{pri-u2} and the proof of Proposition \ref{prop-spa}, we have
\begin{align*}
&\E\big[\|(I-\mathbb L)Y(t)\|_{L^{6}}^{2p}\big]+\E\Big[\Big(\int_0^T\|(-A)(I-\mathbb L)Y(s)\|^2ds\Big)^{p}\Big]\\
\le &C(T,X_0,p).
\end{align*}

Now it suffices to deduce the optimal regularity of $(I-\mathbb L)Y$. 
From the mild form of $(I-\mathbb L)Y(t)$ for equation \eqref{vspde}, it follows that if  $\gamma<2$, then 
\begin{align*}
&\|(I-\mathbb L)Y(t)\|_{\mathbb \HH^{\gamma}}\\
&\le \|e^{-A^2t}(I-\mathbb L)X_0\|_{\HH^{\gamma}}
+\int_0^t\big\|e^{-A^2(t-s)}A(I-\mathbb L)f(Y(s)+Z(s))\big\|_{\HH^{\gamma}}ds\\
&\le C\|X_0\|_{\HH^{\gamma}}
+C\int_0^t(t-s)^{-\frac 12}\big\|e^{-\frac 12A^2(t-s)}(I-\mathbb L)f(Y(s)+Z(s))\big\|_{\HH^{\gamma}}ds\\
&\le C\|X_0\|_{\HH^{\gamma}}
+C\int_0^t(t-s)^{-\frac 12-\frac {\gamma}4}\big(1+\|(I-\mathbb L)Y(s)\|_{L^6}^3+\|\mathbb LY(s)\|^3_{L^6}
\\
&\qquad+\|Z(s)\|_{L^6}^3\big)ds.
\end{align*}
By taking the $p$th moment and making use of the a priori estimates of  $\|\mathbb LY(s)\|_{L^6}$, $\|(I-\mathbb L)Y(s)\|_{L^6}$ and $\|Z(s)\|_{\HH^{\gamma}}$, we finish the proof for the case $\gamma<2$. For the case $\gamma \ge2$, repeating  the above arguments and using similar arguments in the proof of Lemma \ref{sob}, we obtain for a small $\epsilon>0$,
\begin{align*}
&\|(I-\mathbb L)Y(t)\|_{\mathbb \HH^{\gamma}}\\
&\le C\|X_0\|_{\HH^{\gamma}}
+C\int_0^t(t-s)^{-1+\frac {\epsilon}4}\big\|(I-\mathbb L)f(Y(s)+Z(s))\big\|_{\HH^{\gamma-2+\epsilon}}ds\\
&\le C\|X_0\|_{\HH^{\gamma}}
+C(X_0)\int_0^t(t-s)^{-1+\frac {\epsilon}4}
(1+\|X(s)\|_{W^{[\gamma-2+\epsilon],\infty}}^2)
 \|X(s)\|_{H^{\gamma-2+\epsilon}}ds.
\end{align*}
If $\gamma\in [2,4)$, from the estimate of $\|(I-\mathbb L)X(s)\|_{\HH^{\beta}}$, $\beta<2$, it only suffices to bound the moment of $\|X(s)\|_{W^{1,\infty}}$.
By the Sobolev embedding theorem, we have for some small $\epsilon_1$ and any $p\ge 1$.
\begin{align*}
\|X(s)\|_{L^p(\Omega;{W^{1,\infty}})}\le C\|X(s)\|_{L^p(\Omega; H^{\frac 32+\epsilon_1})}\le C(T,X_0,p).
\end{align*}
If $\gamma=4$,  based on the above estimate for the case $\gamma <4$,  we only need to estimate $\|X(s)\|_{W^{2,\infty}}$. By using the Sobolev embedding theorem, we obtain 
\begin{align*}
\|X(s)\|_{L^p(\Omega;{W^{2,\infty}})}\le C\|X(s)\|_{L^p(\Omega; H^{\frac 52+\epsilon_1})}\le C(T,X_0,p).
\end{align*}
Combining the estimates in all cases of $\gamma$, we complete the proof of the optimal spatial regularity estimate \eqref{reg-spa1}.

Now we are in the position to show the temporal regularity estimate  \eqref{reg-tm1}. 
It is not difficult to get that  for $s\le t$, $\gamma\le 4$,
\begin{align*}
\E\Big[\big\|Z(t)-Z(s)\big\|^p\Big]&\le C(T,p)(t-s)^{\min(\frac {\gamma } 4,\frac 12)p}.
\end{align*}
Similar arguments in the proof of Proposition \ref{prop-tm}
yield that if  $\gamma<2$,
\begin{align*}
\E\Big[\|Y(t)-Y(s)\|^p\Big]
\le C(T,X_0,p)(t-s)^{\frac {\gamma p} 4}.
\end{align*}
It suffices to prove that for the case $\gamma\in [2,4]$, we have
\begin{align*}
\E\Big[\|(I-\mathbb L)(Y(t)-Y(s))\|^p\Big]
\le C(T,X_0,p)(t-s)^{\frac {p} 2}.
\end{align*}
The mild form of $(I-\mathbb L)Y$ yields that
for a small $\epsilon_1>0$
\begin{align*}
&\|(I-\mathbb L)(Y(t)-Y(s))\|\\
&\le C\|(I-\mathbb L)X_0\|_{\HH^{\gamma}}(t-s)^{\frac \gamma4}
+\int_s^t (t-s)^{-\frac 12}\|f(Y(r)+Z(r))\|dr\\
&\quad+\int_0^s(s-r)^{-1+\frac \epsilon 2}\|A^{-1+\epsilon}(e^{-A^2(t-s)}-I)A^{-\frac 12}\|\|A^{\frac 12}(I-\mathbb L)f(Y(r)+Z(r))\|dr\\
&\le C\Big(\|(I-\mathbb L)X_0\|_{\HH^{\gamma}}(t-s)^{\frac \gamma4}+(t-s)^{\frac 12}\big(1+\sup_{r\in[0,T]}\|Y(r)\|_{L^6}^3+\sup_{r\in[0,T]}\|Z(r)\|_{L^6}^3\big) \Big)\\
&\quad+C(t-s)^{\frac 34-\frac \epsilon2}
\big(1+\sup_{r\in[0,T]}\|Y(r)\|_{H^1}^3+\sup_{r\in[0,T]}\|Z(r)\|_{H^1}^3\big).
\end{align*}
By taking the $p$th moment and using the a priori estimates of $Y$ and $Z$, we get 
\begin{align*}
\E\Big[\big\|(I-\mathbb L)(Y(t)-Y(s))\big\|^p\Big]&\le C(T,X_0,p)(t-s)^{\frac {p} 2}.
\end{align*}
Combining all the above estimates, we complete the proof.
\end{proof}

\begin{prop}
Let $d=1$, $\gamma\in (0,4]$, $\|A^{\frac {\gamma-2}2}Q^{\frac 12}\|_{\LL_2}<\infty$ and $(I-\mathbb L) X_0\in \HH^{\gamma}$. Suppose that $Q$ commutes with $A$ or $\gamma>\frac 12$. Then we have 
\begin{align*}
\big\|X^N(t)-X(t)\big\|_{L^p(\Omega;{H})}
\le C(T,X_0,p)\lambda_N^{-{\frac \gamma 2}},
\end{align*}
where $C(T,X_0,p)$ is some positive constant and $t\in[0,T]$.
\end{prop}
\begin{proof}
Since $Q$ commutes with $A$ or $\gamma>\frac 12$, the procedures in the proof of  Theorem \ref{tm-strong0} yields the spatial error estimate in the case that $\gamma \in (0,2)$. In the case that $\gamma\in [2,4]$, using  the arguments in the proof of Theorem \ref{tm-strong0} and combining with Corollary \ref{rk-sob}, we get the desired result.
\end{proof}

Next, we present the convergence result of the full discretization in one dimension case.
To simplify the procedures of the proof, we assume that $X_0$ is smooth enough.

\begin{prop}\label{d1-str}
Let $d=1$, $\|A^{\frac {\gamma-2}2}Q^{\frac 12}\|_{\LL_2}<\infty$, $\gamma\in (0,4]$,  $(I-\mathbb L) X_0\in \HH^{\alpha}$ with $\alpha>4$ and $p\ge1$. Assume that $Q$ commutes with $A$ or that $\gamma>\frac 32$. Then there 
exist $\delta t_0\le 1$ and  $ C(X_0,T,p)>0$ such that 
for any $\delta t\le \delta t_0$, $N\in \N^+$,
the numerical solution $X^N_k$, $k\le K$ is strongly convergent to $X$  and  satisfies 
\begin{align*}
\big\|X^N_k-X(t_k)\big\|_{L^p(\Omega;H)}
&\le C(X_0,T,p)(\delta t^{\min(\frac \gamma 2 -\epsilon,1)}+\lambda_N^{-{\frac \gamma 2}})
\end{align*}
for any sufficient small $\epsilon>0$.
\end{prop}

\begin{proof}
Following the procedures in the proofs of Lemmas \ref{lm-ynk} and  \ref{lm-ynk1}, 
it suffices to deduce the temporal strong convergence rate of $\|Y^N(t_k)-\widetilde Y^N_k\|$.
By repeating the procedures in the proof of Lemma
\ref{lm-ynk}, and using Corollary \ref{rk-sob} and Remark \ref{rk-sob1}, we have that
\begin{align*}
\big\|Y^N(t_k)-\widetilde Y^N_{k}\big\|_{L^p(\Omega;\HH)}
&\le C(X_0,T,p)\delta t^{\min(\frac \gamma 2,1)}.
\end{align*}
Then combining  the arguments in the proofs of Theorem \ref{tm-strong1}
and Proposition \ref{prop-str},  we complete the proof.
\end{proof}

In the case that 
$d=2, 3$, $\|A^{\frac 12}Q^{\frac 12}\|_{\LL_2}<\infty$,  we can also obtain the strong convergence rate of the proposed method.
Since the deterministic Cahn--Hilliard equation defines a gradient flow in $\HH^{-1}$ for the energy 
functional $J(u)=\frac 12\|\nabla u\|^2+\int_{\mathcal O}F(u)dx, u\in H^1$, we have the boundedness of $J$ in this case (see, e.g., \cite[Theorem 3.1]{KLM11}).
Then we follow the arguments in the proof of Propositions \ref{prop-spa} and \ref{prop-tm} 
and obtain the following optimal regularity estimates.

\begin{lm}\label{reg-tr1}
Let $d=2,3$, $\|A^{-1+\frac \gamma 2}Q^{\frac 12}\|_{\LL_2}<\infty$, $\gamma\in [3,4]$, $(I-\mathbb L) X_0\in \HH^{\gamma}$ and $p\ge1$.
Then
the unique mild solution  $X$  of equation \eqref{spde}  satisfies 
\begin{align*}
\sup_{t\in [0,T]}\E\Big[\big\|(I-\mathbb L)X(t)\big\|_{\HH^{\gamma}}^p\Big]
&\le C(X_0,T,p),
\end{align*}
and 
\begin{align*}
\E\Big[\big\|X(t)-X(s)\big\|^p\Big]
&\le C(X_0,T,p)(t-s)^{\frac { p} 2}
\end{align*}
for a positive constant $ C(X_0,T,p)$ and $0\le s\le t\le T$.
\end{lm}
\begin{proof}
From \cite[Theorem 3.1]{KLM11} and the Sobolev embedding theorem, it follows that 
$\E\big[\sup\limits_{t\in[0,T]}\|(I-\mathbb L)Y(t)\|_{L^{6}}^{2p}\big]\le C(T,X_0,p)$. 
Then the arguments in Lemma \ref{reg-tr}, together with the Sobolev embedding theorem, yields the desired results. 
\end{proof}

Based on the above regularity estimates, we give the following strong convergence error estimate of the proposed numerical scheme. 
Since 
the Sobolev embedding $L^{\infty}\hookrightarrow H^1$ does not hold,  we make use of the regularity estimates of  exact and numerical solutions in this case.

\begin{prop}
Under the condition of Lemma \ref{reg-tr1}, there
exist $\delta t_0\le 1$ and  $ C(X_0,T,p)>0$ such that 
for any $\delta t\le \delta t_0$, $N\in \N^+$, the numerical solution $X^N_k$, $k\le K$ satisfies 
\begin{align*}
\big\|X^N_k-X(t_k)\big\|_{L^p(\Omega;{H})}
&\le C(X_0,T,p)(\delta t+\lambda_N^{-\frac \gamma 2}).
\end{align*}
\end{prop}
\begin{proof}
Following the proof of Theorem \ref{tm-strong0},
we have 
\begin{align*}
&\big\|Y^N(t)-P^NY(t)\big\|^2_{L^{2p}(\Omega;{\HH^{-1}})}
+\big\|\int_0^{t}\|\nabla (P^NY(s)-Y^N(s))\|^2ds\big\|_{L^{p}(\Omega;\R)}\\
&\le C(T,X_0,p)\lambda_N^{- \gamma }.
\end{align*} 
By the mild form of $Y^N(t)$ and $P^NY(t)$, we get 
\begin{align*}
\|Y^N(t)-P^NY(t)\|
&\le \|\int_0^te^{-A^2(t-s)}A(f(Y^N+Z^N)-f(Y+Z))ds\|
\\
&\le C\int_0^t(t-s)^{-\frac 12}\int_0^1\Big\|(f'(\theta (Y^N+ Z^N)
+(1-\theta)(Y+Z))\\
&\qquad ((I-P^N)Z+(I-P^N)Y)d \theta 
\Big)\Big\|ds\\
&\quad +C \int_0^t(t-s)^{-\frac 14}\|(f'(\theta (Y^N+ Z^N)
+(1-\theta)(Y+Z))\\
&\qquad(Y^N(t)-Y(t))\|_{H^1}d\theta ds.
\end{align*}
From the regularity estimates in Lemma \ref{reg-tr1}, it follows that 
\begin{align*}
\big\|Y^N(t_k)-P^NY(t_k)\big\|_{L^p(\Omega;{H})}\le C\lambda_N^{-\frac \gamma2}.
\end{align*}
The arguments in the proofs of Lemmas \ref{lm-ynk} and  \ref{lm-ynk1} lead to 
\begin{align*}
&\|\widetilde Y^N_k-Y^N_k\|^2_{L^{2p}(\Omega;\HH^{-1})}
+\big\|\int_0^{t_k}\|\nabla (\widetilde Y^N_{[s]+1}-Y^N_{[s]+1})\|^2 ds\big\|_{L^{p}(\Omega;\R)}\\
&\le C(T,X_0,p)\delta t^2.
\end{align*}
Then repeating the above procedures  for spatial error
estimate and the arguments in Theorem \ref{tm-strong1}, we complete the proof.
\end{proof}

\section{Appendix}

\textit{Proof of Lemma \ref{sob}}:
	The equivalence of norms in $\HH^{\beta}$ and 
	$H^{\beta}$ and H\"older inequality  yield that for $z\in \HH^{\beta}\cap E$, $\beta<1$,
	\begin{align*}
	&\|(I-\mathbb L)g(x)z\|_{\HH^{\beta}}^2\\
	&\le C\|(I-\mathbb L)g(x)z\|^2
	+C\int_0^L\int_0^L\frac {|g(x(\xi_1))z(\xi_1)-g(x(\xi_2))z(\xi_2)|^2}{|\xi_1-\xi_2|^{2\beta+1}}d\xi_1d\xi_2\\
	&\le C(1+\|x\|_E^4)\|z\|^2
	+C(1+\|x\|_E^4)\|z\|_{\HH^{\beta}}^2
	\\
	&\quad+C(1+\|x\|_E^2)\|(I-\mathbb L)x\|_{\HH^{\beta}}^2\|z\|_{E}^2\\
	&\le C(1+\|x\|_E^4+\|(I-\mathbb L)x\|_{\HH^{\beta}}^4)(\|z\|_{\HH^{\beta}}^2+\|z\|_{E}^2).
	\end{align*}
	According to the self-adjointness of $g$ and the Sobolev embedding theorem, we have 
	\begin{align*}
	&\|(I-\mathbb L)g(x)y\|_{\HH^{-1}}=\sup_{\|w\|_{\HH}\le 1}\Big|\<y,(I-\mathbb L)g(x)A^{-\frac 12}w\>\Big|\\
	%&= \sup_{\|w\|_{\HH}\le 1}\Big|\<A^{-\frac \beta 2}y,A^{\frac \beta 2}(I-\mathbb L)g(x)A^{-\frac 12}w\>\Big|\\
	&\le  \sup_{\|w\|_{\HH}\le 1}\|y\|_{\HH^{-\beta}}
	\|(I-\mathbb L)g(x)A^{-\frac 12}w\|_{\HH^{\beta}}\\
	&\le C\sup_{\|w\|_{\HH}\le 1}\|y\|_{\HH^{-\beta}}
	(1+\|x\|_E^2+\|(I-\mathbb L)x\|_{\HH^{\beta}}^2)
	(\|w\|_{\HH^{\beta-1}}+\|A^{-\frac 12}w\|_E)\\
	&\le C(1+\|x\|_E^2+\|(I-\mathbb L)x\|_{\HH^{\beta}}^2)\|y\|_{\HH^{-\beta}},
	\end{align*}
	which completes the proof.

\bibliographystyle{amsplain}
\bibliography{bib}

\end{document}